\documentclass[10pt,journal,compsoc]{IEEEtran}
\ifCLASSOPTIONcompsoc
  \usepackage[nocompress]{cite}
\else
  \usepackage{cite}
\fi
\ifCLASSINFOpdf
\else
\fi

\usepackage{amsthm}
\usepackage{amsmath}
\usepackage{mathtools}
\usepackage{amssymb}
\usepackage{bbm}
\usepackage{enumitem}
\usepackage{multirow}
\usepackage{makecell}
\usepackage{subcaption}
\usepackage{xcolor}
\usepackage{algorithm,algorithmic}
\usepackage[normalem]{ulem}
\usepackage{url}
\usepackage{graphicx}

\newtheorem{theorem}{Theorem}[section]
\newtheorem{definition}[theorem]{Definition}
\newtheorem{corollary}[theorem]{Corollary}
\newtheorem{lemma}[theorem]{Lemma}
\newtheorem{example}[theorem]{Example}
\newtheorem{proposition}[theorem]{Proposition}
\newtheorem{remark}[theorem]{Remark}

\newcommand{\R}{\mathbb{R}}

\newcommand{\Pc}{\mathcal{P}_c}

\newcommand{\norm}[1]{\left\lVert#1\right\rVert}
\newcommand{\abs}[1]{\left\vert#1\right\vert}

\newcommand{\sgn}{\textnormal{sgn}}

\begin{document}
%
\title{On the Treatment of Optimization Problems with L1 Penalty Terms via Multiobjective Continuation}
%
%
%
%

\author{Katharina Bieker,
        Bennet Gebken
        and~Sebastian Peitz
\IEEEcompsocitemizethanks{\IEEEcompsocthanksitem K. Bieker and B. Gebken are with the Department
of Mathematics, Paderborn University, Germany.\protect\\
E-mail: bieker@math.upb.de\protect\\
\IEEEcompsocthanksitem S. Peitz is with the Department of Computer Science, Paderborn University, Germany.
}%
\thanks{~}}

%
%

\markboth{~}%
{Shell \MakeLowercase{\textit{et al.}}: Bare Demo of IEEEtran.cls for Computer Society Journals}
%



\IEEEtitleabstractindextext{%
\begin{abstract}
We present a novel algorithm that allows us to gain detailed insight into the effects of sparsity in linear and nonlinear optimization. Sparsity is of great importance in many scientific areas such as image and signal processing, medical imaging, compressed sensing, and machine learning, 
as it ensures robustness against noisy data and yields models that are easier to interpret
due to the small number of relevant terms. It is common practice to enforce sparsity by adding the $\ell_1$-norm as a penalty term. 
In order to gain a better understanding and to allow for an informed model selection, we directly solve the corresponding multiobjective optimization problem (MOP) that arises when minimizing the main objective and the $\ell_1$-norm simultaneously. 
As this MOP is in general non-convex for nonlinear objectives, the penalty method will fail to provide all optimal compromises. To avoid this issue, we present a continuation method specifically tailored to MOPs with two objective functions one of which is the $\ell_1$-norm. Our method can be seen as a generalization of homotopy methods for linear regression problems to the nonlinear case.
Several numerical examples -- including neural network training -- demonstrate our theoretical findings and the additional insight gained by this multiobjective approach.
\end{abstract}
%
}
%
%
\maketitle
%
\IEEEdisplaynontitleabstractindextext
%
%
\IEEEpeerreviewmaketitle
\section{Introduction}
Optimization problems are at the center of a very large number of problems in science and engineering. In many situations, the task of minimizing some specified objective function is further complicated by the desire to enforce sparsity,
i.e., to obtain a solution with as few non-zero entries as possible,
which has several important advantages. Most notably,
sparse solutions tend to be robust against noise. 
Furthermore, they often allow for a better analysis and interpretation due to the small number of relevant terms, for instance in the identification of governing equations \cite{BPK16}. Practical applications can be found in many different areas, e.g., signal processing, compressed sensing, neural network training and medical imaging \cite{EK12,SurveyL1,SpecialIssueSparse,Elad2010,Plumbley2010}.

In signal and image processing, one frequently considers (overdetermined) linear regression models
with an $\ell_1$ penalty term:
\begin{align}\label{eq:lin_penalty}
\min_{x \in \R^n} \norm{Ax-b}_2^2 + \lambda \norm{x}_1,
\end{align}
where $A \in \R^{m\times n}$, $b \in \R^m$ and $\lambda \in \R^{\ge0}$. This problem is often referred to as \emph{Lasso} \cite{Lasso} or \emph{basis pursuit} \cite{BasisPursuit}.
Since appropriately weighting the penalty term is difficult a priori, the solution is computed for multiple values of $\lambda$, for instance in \emph{homotopy methods}~\cite{osborne2000,malioutov2005,bringmann2016,donoho2008}, where the entire solution path for $\lambda \in \R^{\ge0}$ is tracked. The method is based on the subdifferential of the $\ell_1$-norm and the resulting Karush-Kuhn-Tucker (KKT) conditions, while exploiting the fact that the structure of zero entries along the solution path is locally constant almost everywhere.

In recent years, nonlinear models have become increasingly popular, in particular due to the advances in deep learning and the associated neural network training, where sparsity can help to avoid overfitting \cite{Goodfellow2016}. As in the linear case, the use of $\ell_1$ penalty terms is a natural and common choice for obtaining sparse models, resulting in the problem
\begin{equation}\label{eq:nonlin_penalty}
\min_{x\in\R^n} f(x) + \lambda \|x\|_1 \quad \text{with} \quad \lambda \in \R^{\ge0}.
\end{equation}
Since $f$ is based on a nonlinear model and is generally non-convex, \eqref{eq:nonlin_penalty} is much more difficult to solve than \eqref{eq:lin_penalty}.

In contrast to the classical $\ell_1$ penalty approach, we will here interpret sparse optimization as a (non-smooth) \emph{multiobjective optimization problem} (MOP), where the objectives $f$ and $\norm{\cdot}_1$ are minimized simultaneously, i.e.,
\begin{equation}
\tag{MOP-${\ell_1}$}
\min_{x \in \R^n} \left( \begin{array}{c}
f(x)  \\
\norm{x}_1
\end{array}\right).
\label{eq:MOP_l1}
\end{equation}
In this context, \eqref{eq:nonlin_penalty} is equivalent to the \emph{weighted-sum approach}, i.e., $\min_{x \in \R^n} \alpha_1 f(x) + \alpha_2 \|x\|_1$ with $\alpha_1, \alpha_2 \in [0,1]$, $\alpha_1 + \alpha_2 = 1$ \cite{Miettinen2004}.
It is well-known that every solution of the weighted-sum approach is a solution of the MOP \cite{Miettinen2004}. 
As shown in Figure~\ref{fig:weighted_sum}, the weight vector $\alpha$ is orthogonal to the \emph{Pareto front}. Varying $\alpha$ thus yields different compromise solutions. However, if $f$ is non-convex, then $\alpha$ does no longer uniquely define a single Pareto optimal point, and we cannot find points in the non-convex parts of the Pareto front, see \cite{Ehrgott2005,Miettinen2004} for in-depth discussions.
This is also relevant in practice, e.g., in neural network training. For nonconvex loss functions, well generalizing sparse solutions may not be computable via the penalty approach. This is demonstrated in Section~\ref{subsec:neural_network}.
\begin{figure}
	\centering
	\parbox[b]{0.15\textwidth}{\centering \includegraphics[width=0.15\textwidth]{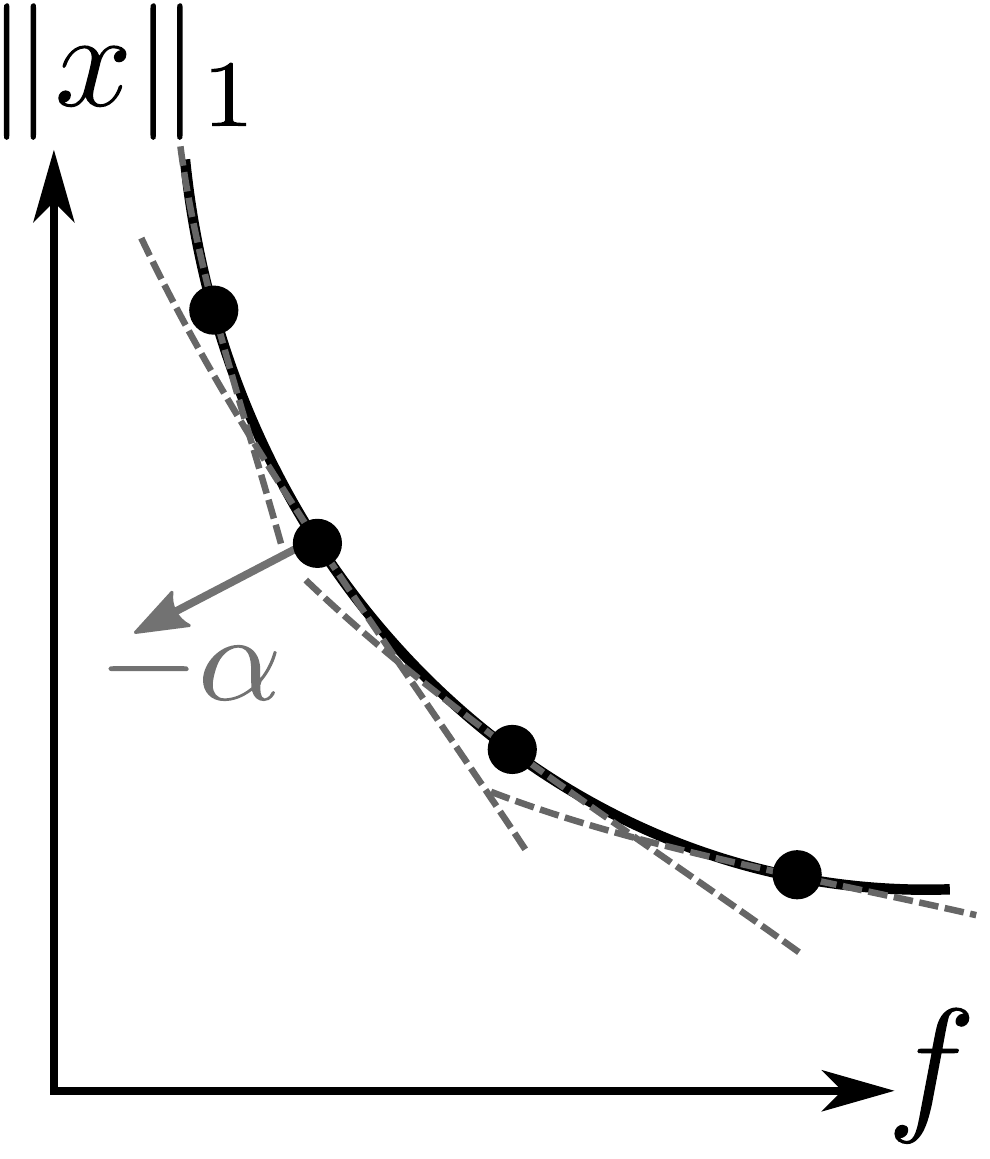} \\ (a)} \hfil
	\parbox[b]{0.15\textwidth}{\centering \includegraphics[width=0.15\textwidth]{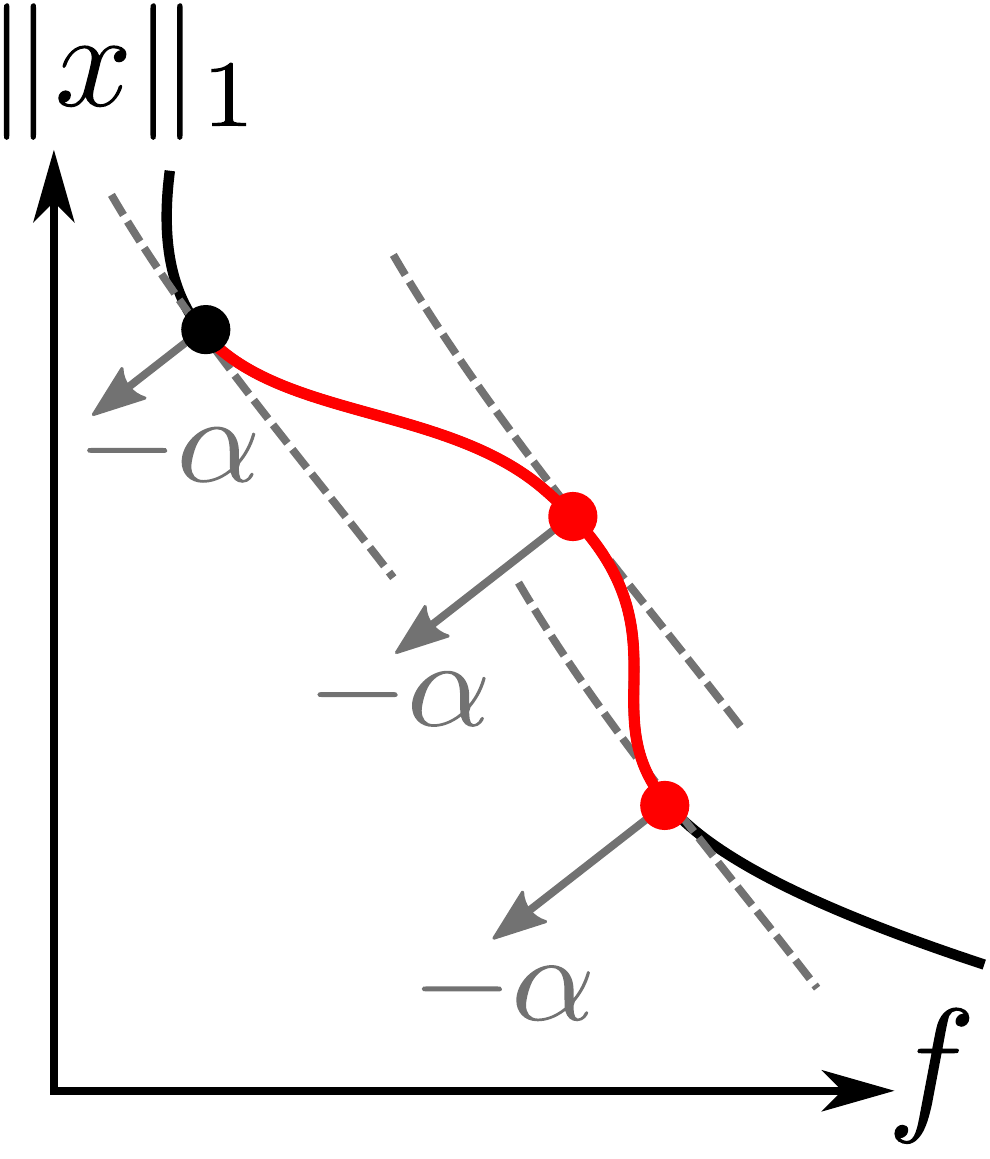} \\ (b)} \hfil
	\parbox[b]{0.15\textwidth}{\centering \includegraphics[width=0.15\textwidth]{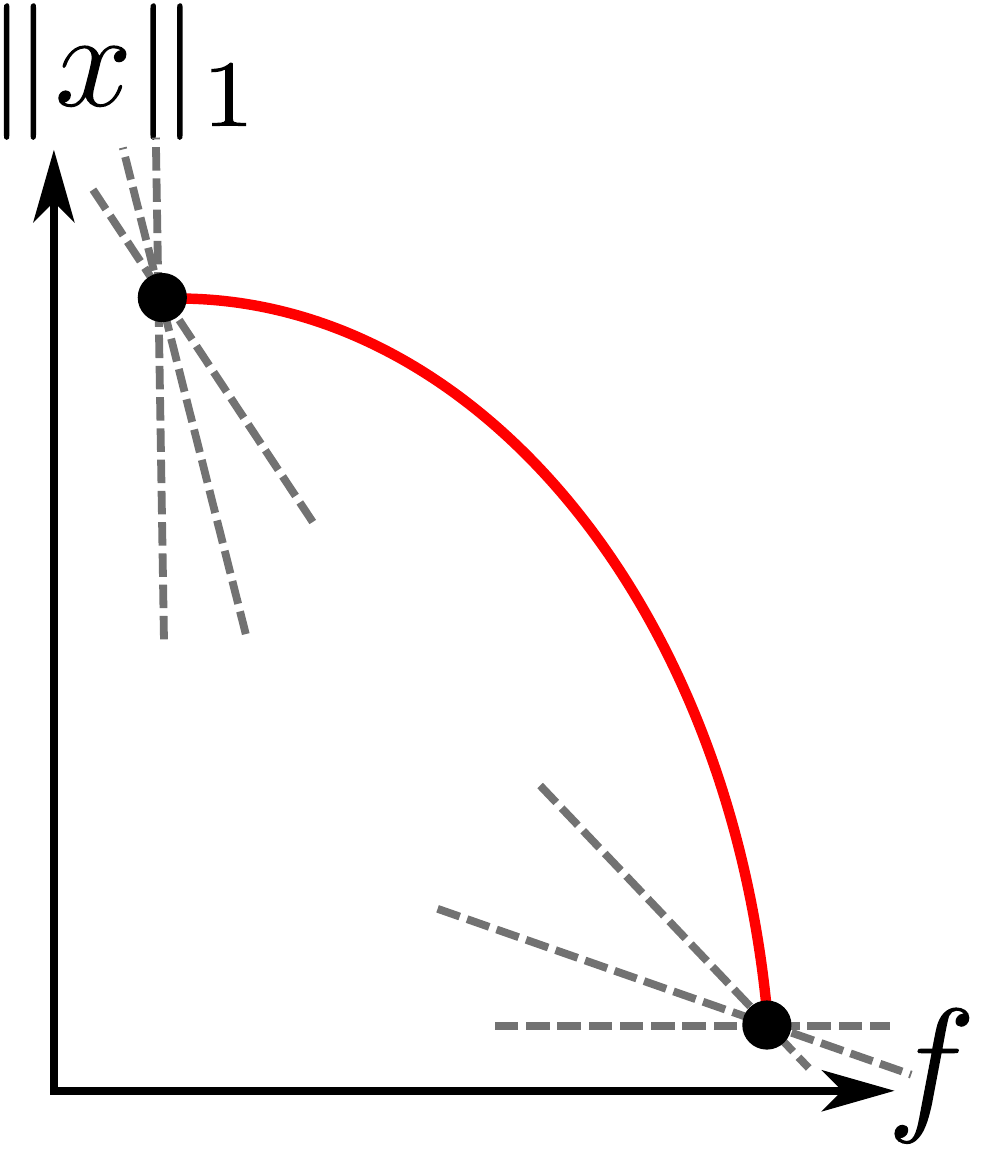} \\ (c)}
	\caption{(a) Different points on the Pareto front by solving \eqref{eq:MOP_l1} using the weighted-sum approach with varying $\alpha$. (b) For non-convex problems, multiple solutions correspond to the same $\alpha$ such that the solution cannot be computed in the non-convex regions. 
	(c) For strictly concave problems, only the individual minima can be found.}
	\label{fig:weighted_sum}
\end{figure}

Alternative solution approaches suitable for non-convex problems are \emph{evolutionary algorithms} \cite{Deb2001,Deb2002}, \emph{scalarization methods} \cite{Miettinen2004}, as well as \emph{continuation methods} \cite{Hillermeier2001,ParetoTracer,SCM+20}.
The latter require sufficiently smooth objective functions ensuring the existence of a tangent space as the idea is to compute a discretization of the Pareto critical set by first performing a so-called \emph{predictor step} along the tangent space and then computing the nearest Pareto critical point using an optimization procedure (\emph{corrector step}).

In this paper, we present a new continuation method for non-smooth MOPs with two objectives, where the first objective is twice continuously differentiable and the second objective is the $\ell_1$-norm (Section~\ref{sec:PC}). 
In particular, the first objective may be non-convex. 
Although the solution set of non-smooth MOPs does not generally possess the smooth structure required for classical continuation, we will show that this specific problem class is piecewise-smooth. 
Thus, the solution set can be computed via classical continuation up to ``kinks''. 
We will show how these kinks can be detected and addressed within the continuation by exploiting the special structure of the optimality conditions specifically derived for this purpose. The global minimal point $0 \in \R^n$ of the $\ell_1$-norm can be used as the starting point. 

Compared to scalarization methods (like the \emph{Adam descent method} \cite{KB14} applied to \eqref{eq:nonlin_penalty}), our approach has the advantage that we compute an approximation of the entire Pareto set instead of just a single point. In particular, we do not have to choose a penalty parameter. Compared to evolutionary methods (which also approximate the entire Pareto set), our approximation is generally closer to the true Pareto set while requiring fewer function evaluations. Furthermore, our method detects structural properties like kinks or the connectedness of the Pareto set, which remain unnoticed in (standard) evolutionary methods. On the other hand, some of the disadvantages and challenges of our method will be discussed in Section~\ref{sec:high_dimensional}.

Beyond that, alternative approaches at the intersection of sparse and multiobjective optimization are, for instance, the training of (linear) SVM models using $\varepsilon$-constraint scalarization \cite{acskan2014,aytug2012}, or the solution of MOPs with linear models and the $\ell_1$-norm via evolutionary algorithms \cite{Li2012}.
In the context of sparse optimization, generalizations (for convex problem classes) of the classical homotopy methods were introduced in \cite{park2007,Rosset2004,rosset2005,Zhao2004}, for instance to arbitrary convex loss-functions or convex penalty terms.
Thus, the main advantage of our method is the treatment of nonlinear, non-convex problems, which will be demonstrated using several examples of varying complexity (Section~\ref{sec:numerics}).

\section{Multiobjective and non-smooth optimization}  
\label{sec:overview_optimization}

Before presenting the solution approach, we give a short overview of multiobjective and non-smooth optimization. For more details, see \cite{Miettinen2004, Ehrgott2005, Hillermeier2001} and \cite{Clarke1990,Bagirov2014}, respectively. 

The general (unconstrained) multiobjective optimization problem is to minimize multiple (real-valued) functions $f_i:~\R^n \rightarrow \R$, $i \in \{1,...,k\}$, at the same time:
\begin{equation}
\tag{MOP}
    \min_{x \in \R^n} F(x) = \min_{x \in \R^n} \left( \begin{array}{c}
    f_1(x)  \\
    \vdots \\
    f_k(x)
    \end{array}\right).
\label{eq:MOP}
\end{equation}

Due to the loss of a total order of $\R^k$ for $k>1$, the solution to \eqref{eq:MOP} is not a specific point, but the \emph{set of non-dominated points}, also referred to as the \emph{Pareto set}. Non-dominated points $x^*$ are characterized by the fact that there exists no $x$ with $f_i(x) \le f_i(x^*)$ for all $i \in \{1, \dots, k\}$ and $f_j(x) < f_j(x^*)$ for at least one $j \in \{1, \dots, k\}$. The image of the Pareto set under $F$ is known as the \emph{Pareto front}, lying in the \emph{objective space} $\R^k$.

If all objectives are continuously differentiable, then these points satisfy the well-known necessary KKT conditions, which state that if a point $x^*$ is locally Pareto optimal, there exists a convex combination of the individual gradients that is zero \cite{Hillermeier2001}, i.e., there exists $\alpha^* \in (\R^{\ge0})^k$ with
\begin{align*}
    \sum_{i=1}^k \alpha_i^* \nabla f_i(x^*) = 0 \text{ and } \sum_{i = 1}^k \alpha_i^* = 1.
\end{align*}

Points satisfying the KKT conditions are denoted as \emph{Pareto critical}, and the set containing all Pareto critical points is the \emph{Pareto critical set}, which is a superset of the Pareto set.

As the $\ell_1$-norm does not satisfy the required smoothness assumptions for the classical KKT conditions, we need to consider subgradients instead.

\begin{definition}[Subdifferential and subgradient]
	Let $f:~\R^n \rightarrow \R$ be locally Lipschitz continuous, and let $\Omega \subset \R^n$ be the set of points where $f$ is not differentiable. Then
	\begin{align*}
    	\partial f(x) := \mathbf{Conv} (\{ &g \in \R^n~|~\exists (x_j)_j \in \R^n \setminus \Omega \text{ with } x_j \rightarrow x \\
    	&\text{and } \nabla f(x_j) \rightarrow g \text{ for } j \rightarrow \infty \} )
	\end{align*}
	is the \emph{(Clarke) subdifferential of} $f$ \emph{in} $x$, where $\mathbf{Conv}$ denotes the convex hull. An element $g \in \partial f(x)$ is called a \emph{subgradient}. 
\end{definition}

If $f$ is continuously differentiable in $x$, then the subdifferential is the set containing the gradient, i.e., $\partial f(x) = \{\nabla f(x)\}$. 
In case of the $\ell_1$-norm, the subgradients take the following form. 

\begin{example}[Subgradients and subdifferential of the $\ell_1$-norm]\label{ex:l1_norm}
	The $\ell_1$-norm in $\R^n$,
	$f_{\ell_1}(x) = \norm{x}_1 = \sum_{i=1}^n |x_i|$, 
	is non-differentiable in all  $x$ with $x_j = 0$ for some $j \in \{1,\dots,n\}$. 
	To calculate $\partial f_{\ell_1}(x)$ for an arbitrary $x \in \R^n$, we first distinguish between \textbf{\emph{active}} and \textbf{\emph{inactive}} indices, with $\{j \in J ~|~x_j \neq 0\}=:A(x)$ being the \textbf{\emph{active set}}.
	It is easy to see that
	\begin{equation*}
    	g = \left( \begin{array}{c} g_1 \\ \vdots \\ g_n \end{array} \right) \in \partial f_{\ell_1}(x) \Leftrightarrow \left\{
    	\begin{array}{ll}
    	g_j = \sgn(x_j) & \text{if } j \in A(x),\\
    	g_j \in [-1,1] & \text{if } j \notin A(x).
    	\end{array} 
    	\right.
	\end{equation*}
	
	We denote the number of active indices by $n^A(x) := \abs{A(x)}$ and the number of inactive indices by $n^0(x) :=  n - n^A(x)$.
	The subdifferential is now the convex hull of $M(x) := 2^{n^0(x)}$ vectors $g^i$ containing the extreme points of the subdifferential:
	\begin{equation*}
	    \partial f_{\ell_1}(x) = \mathbf{Conv}\left(\left\{g^i ~|~ i \in \{1,\dots,M(x)\}\right\}\right) \text{, i.e.,}
	\end{equation*}
	\[
	g^i_j = \sgn(x_j) ~\forall j \in A(x) \quad \text{and} \quad g^i_j \in \{-1,1\} ~\forall j \notin A(x).
	\]
	
	Consider the point $\hat{x} = (2, 0, -3, 0)^\top \in \R^4$ for example. Then we have $A(\hat x) = \{1, 3\}$, $n^A(\hat x) = 2$, $n^0(\hat x) = 2$, $M(\hat x) = 4$, and
	\begin{equation*}
    	\partial f_{\ell_1}(\hat{x}) = \mathbf{Conv}\left( \left\{\begin{pmatrix}1 \\ 1 \\ -1 \\ -1\end{pmatrix}, \begin{pmatrix}1 \\ -1 \\ -1 \\ -1\end{pmatrix}, \begin{pmatrix}1 \\ 1 \\ -1 \\ 1\end{pmatrix}, \begin{pmatrix}1 \\ -1 \\ -1 \\ 1\end{pmatrix} \right\}\right).
	\end{equation*}
\end{example}

Analogous to the smooth case, necessary optimality conditions for MOPs containing non-smooth objective functions can be obtained \cite{Makela2014}.

\begin{proposition}[Non-smooth KKT conditions]
	Assume that $F : \R^n \rightarrow \R^k$ is locally Lipschitz continuous and $x^*$ is locally Pareto optimal. Then
	\begin{equation}
    	0 \in \mathbf{Conv} \left( \bigcup_{i = 1}^k \partial f_i(x^*) \right) \subset \R^n.
    	\label{eq:KKT_nonsmooth}
	\end{equation}
\end{proposition}  

Similar to the smooth case, we call $x \in \R^n$ \emph{Pareto critical} if it satisfies the KKT conditions \eqref{eq:KKT_nonsmooth} and we call the set of all those points the \emph{Pareto critical set} $\Pc$.
\section{The Continuation Method}\label{sec:PC}
Our goal is to compute the entire Pareto critical set $\Pc$ of \eqref{eq:MOP_l1} using a continuation method that exploits the structure of the $\ell_1$-norm and its subgradients.
As a starting point we can use $x = 0$ which is always Pareto critical as the minimum of the $\ell_1$-norm, and then move along $\Pc$ in the direction where the value of $f$ decreases and $\norm{x}_1$ increases.

\begin{algorithm}[b]
	\begin{algorithmic}
		\STATE \textbf{Input:} $f:\R^n \rightarrow \R \in \mathcal{C}^2$, $x_{\text{start}} \in \R^n$ satisfying \eqref{eq:KKT_nonsmooth}
		\STATE \textbf{Output:} Discretization of the Pareto critical set $\mathcal{P}_{\textnormal{PC}}$
		
		\STATE $\mathcal{P}_{\textnormal{PC}} = \emptyset$, $x^0 = x_{\text{start}}$
		\WHILE{End of $\Pc$ component has not been reached}
		\STATE $\mathcal{P}_{\textnormal{PC}} = \mathcal{P}_{\textnormal{PC}} \cup \{x^0\}$\\
		\IF{$\Pc$ is locally smooth}
		\STATE Predictor step tangential to $\Pc$: \hfill $x^p$ $\leftarrow$ predictor($x^0$)\\ 
		\STATE Corrector step with $x^c \in \Pc$: \hfill $x^c$ $\leftarrow$ corrector($x^p$)\\ 
		\STATE $x^0 = x^c$\\
		\ELSE
		\STATE (De-)Activate entries\\
		\ENDIF
		\ENDWHILE
	\end{algorithmic}
	\caption{PC algorithm} \label{alg:PC_rough_concept}
\end{algorithm}

Continuation methods (or \emph{predictor-corrector (PC) methods}) yield a sequence of Pareto critical points by -- starting from a Pareto critical point -- performing a prediction step along the tangent space of $\Pc$ and then a correction step which results in the next Pareto critical point close by, see \cite{Hillermeier2001} for details.
For the non-smooth problem \eqref{eq:MOP_l1}, the Pareto critical set may contain kinks such that we cannot exclusively rely on the tangent space in the prediction step. 
Thus, when we reach such a kink, we additionally require a mechanism that computes all possible directions in which the Pareto critical set might continue, see Algorithm~\ref{alg:PC_rough_concept} for a sketch.
To this end, we will first have a more detailed look at the KKT conditions \eqref{eq:KKT_nonsmooth} for the special case of problem~\eqref{eq:MOP_l1} in Section~\ref{subsec:OC_MOPL1} before giving a detailed description of the predictor and corrector steps in Sections~\ref{subsec:predictor} and \ref{subsec:corrector}. 
We then discuss strategies for the (de-)activation of indices (i.e., how to overcome kinks) in Section~\ref{subsec:PC_act_deact}, and subsequently address implementation and globalization in Section~\ref{subsec:Implementation}.
\subsection{Optimality conditions for \eqref{eq:MOP_l1}}
\label{subsec:OC_MOPL1}
Since we assume $f$ to be at least twice continuously differentiable, the gradient $\nabla f(x)$ exists for every $x \in \R^n$. 
The subdifferential of the $\ell_1$-norm is given by the convex hull of $M(x) = 2^{n^0(x)}$ vectors $g^i$ given in Example~\ref{ex:l1_norm}. 
To simplify the notation, we will write $M$ and $n^0$ instead of $M(x)$ and $n^0(x)$.
Inserting this into the KKT conditions \eqref{eq:KKT_nonsmooth}, a point $x$ is Pareto critical for \eqref{eq:MOP_l1} if there exist KKT multipliers $\alpha_1, \alpha_2 \in [0,1]$ with $\alpha_1 + \alpha_2 = 1$ and  
\begin{equation}\label{eq:KKT_l1}
    \alpha_1 \nabla f(x) + \alpha_2 (\beta_1 g^1 + \dots + \beta_{M} g^{M}) = 0,
\end{equation}
where $\beta_i \in [0,1]$ for $i=1,\dots,M$ and $\sum_{i=1}^M \beta_i = 1$.
The following theorem will further simplify this condition. 

\begin{theorem}\label{th:KKT_cond}
	Let $x \in \R^n$ with $x \neq 0$ and let $a \in A(x) = \left\{j \in \{1,\dots,n\}~|~ x_j \neq 0\right\}$ be any active index.
	Then, the following statements are equivalent:
	\begin{enumerate}
		\item \label{th_state_critical} $x$ is Pareto critical for~\eqref{eq:MOP_l1}
		\item \label{th_state_newCond} The following conditions hold for every $j \in \{1,\dots,n\}$:\label{eq:th_KKT_cond}
		\begin{enumerate}
			\item If $j \in A(x)$: 
			\begin{enumerate}
				\item $\sgn((\nabla f(x))_j) = -\sgn(x_j)$ if $(\nabla f(x))_j \neq 0$ and \label{eq:th_KKT_cond_act_1}
				\item $\abs{(\nabla f(x))_j} = \abs{(\nabla f(x))_a}$
				\label{eq:th_KKT_cond_act_2}
			\end{enumerate} 
			\item If $j \notin A(x)$: $\abs{(\nabla f(x))_j} \le \abs{(\nabla f(x))_a}$\label{eq:th_KKT_cond_inact} \\
			
		\end{enumerate}
	\end{enumerate}
\end{theorem}

\begin{proof}
	We start by showing \ref{th_state_critical} $\Rightarrow$ \ref{th_state_newCond}.
	Since $x$ is Pareto critical, there exist $\alpha_1, \alpha_2 \in [0,1]$ with $\alpha_1 + \alpha_2  = 1$ and $\beta_i \in [0,1]$, $i=1,\dots,M$, with $\sum_{i=1}^M \beta_i = 1$ such that \eqref{eq:KKT_l1} is satisfied.
	
	Note that $\alpha_1 = 0$ would imply that zero is contained in the subdifferential of the $\ell_1$-norm at $x$, which is only the case for $x = 0$. Thus, we must have $\alpha_1 > 0$.
	
	First, assume $j \in A(x)$. Then, $(g^i(x))_j = \sgn(x_j)$ for all $i=1,\dots,M$ and 
	\begin{align*} 
		0 &= \alpha_1 (\nabla f(x))_j + \alpha_2 \sgn(x_j) \\
		&= \alpha_1 (\nabla f(x))_j + (1 - \alpha_1) \sgn(x_j).
	\end{align*}
	Since $\alpha_1 > 0$, this is equivalent to
	\begin{align*}
	(\nabla f(x))_j  =  -\frac{(1 - \alpha_1)}{\alpha_1} \sgn(x_j).
	\end{align*}
	Therefore, we have
	\begin{align*}
	\sgn(\nabla f(x))_j = -\sgn(x_j) \quad \Rightarrow \quad \abs{(\nabla f(x))_j}  =  \frac{(1 - \alpha_1)}{\alpha_1}.
	\end{align*}
	
	Since this has to hold for every active $j$, it follows
	\begin{align*}
	\abs{(\nabla f(x))_j} = \abs{(\nabla f(x))_a}.
	\end{align*}
	
	Now assume that $j$ is inactive. Define
	\begin{align*}
    	\beta_{j,-} := \sum_{i : g_j^i = -1} \beta_i, \quad
    	\beta_{j,+}  := \sum_{i : g_j^i = 1} \beta_i.
	\end{align*}
	From \eqref{eq:KKT_l1} we obtain
	\begin{align*}
    	0 \quad &= \alpha_1 (\nabla f(x))_j + \alpha_2(\beta_{j,+} - \beta_{j,-}) \\
    	\Rightarrow \quad \abs{(\nabla f(x))_j} &= \abs{\beta_{j,-} - \beta_{j,+}}\frac{1-\alpha_1}{\alpha_1} \\  &\leq \frac{1 - \alpha_1}{\alpha_1} = \abs{(\nabla f(x))_a}.
	\end{align*}
	
	For the opposite direction "$\Leftarrow$", assume that $x$ satisfies 2(a) and (b). Note that 2(a) implies $\sgn(x_a)(\nabla f(x))_a < 0$ and thus,
	\begin{align*}
	1 - \sgn(x_a)(\nabla f(x))_a > 1.
	\end{align*}
	Define
	\begin{align*}
	\alpha_1 := \frac{1}{1 - \sgn(x_a)(\nabla f(x))_a} \in (0,1).
	\end{align*}
	By 2(a)(i) we obtain 
	\begin{align*}
	\alpha_1 = \frac{1}{1 +\abs{(\nabla f(x))_a}}.
	\end{align*}
	With 2(a)(i) it follows that for all active $j$
	\begin{align*}
    	\frac{\alpha_1}{1 -\alpha_1}\abs{(\nabla f(x))_j} 
    	&= \frac{1}{\frac{1}{\alpha_1} -1}\abs{(\nabla f(x))_a} = 1.
	\end{align*}
	In particular,
	\begin{align*}
    	& 1 = \frac{\alpha_1}{1 -\alpha_1}\abs{(\nabla f(x))_j} = \frac{\alpha_1}{1 -\alpha_1} \sgn(\nabla f(x)_j) \nabla f(x)_j \\
    	\Rightarrow \quad& \frac{\alpha_1}{1 -\alpha_1} \nabla f(x)_j = \sgn(\nabla f(x)_j) = -\sgn(x_j).
	\end{align*}
	
	Analogously, for all inactive $j$, we can use 2(b) to obtain
	\begin{align*}
    	\frac{\alpha_1}{1 -\alpha_1}\abs{(\nabla f(x))_j} \leq 1 \quad
    	\Rightarrow \quad \frac{\alpha_1}{1 -\alpha_1} (\nabla f(x))_j \in [-1,1].
	\end{align*}
	
	As a result, the vector $-\frac{\alpha_1}{1 -\alpha_1} \nabla f(x)$ can be written as a convex combination of the vectors $g^i$ defined in Example~\ref{ex:l1_norm}. Let $\beta_1, ..., \beta_M$ be the corresponding coefficients and define $\alpha_2 = 1 - \alpha_1$. Then
	\begin{align*}
    	&-\frac{\alpha_1}{1 -\alpha_1} \nabla f(x) = \beta_1 g^1 + \dots + \beta_{M} g^{M} \\
    	\Leftrightarrow \quad& \alpha_1 \nabla f(x) + \alpha_2(\beta_1 g^1 + \dots + \beta_{M} g^{M}) = 0,
	\end{align*}
	showing that \eqref{eq:KKT_l1} holds.
	
\end{proof}

Theorem \ref{th:KKT_cond} states that the absolute values of all gradient entries belonging to active indices are identical. At the same time, the absolute values of gradient entries belonging to inactive indices must always be less than or equal to the absolute value of active gradient entries. 
Furthermore, the proof of the theorem yields an explicit formula for the KKT multipliers of Pareto critical points. 

\begin{remark} \label{rem:KKT_multipliers}
	If $x$ is Pareto critical for~\eqref{eq:MOP_l1} and $a$ is an active index, then KKT multipliers corresponding to $x$ are given by
	\begin{align*}
	\alpha_1 = \frac{1}{1 +\abs{(\nabla f(x))_a}}, \quad \alpha_2 = \frac{\abs{(\nabla f(x))_a}}{1 + \abs{(\nabla f(x))_a}}.
	\end{align*}
	If $\nabla f(x) \neq 0$ then the KKT multipliers are unique. In particular, this indicates that ``kinks'' in the Pareto front can only occur if $\nabla f(x) = 0$.
\end{remark}
 
As mentioned before, since \eqref{eq:MOP_l1} is a non-smooth problem, the Pareto critical set will generally contain points in which the tangent space does not exist. Theorem \ref{th:KKT_cond} can be used to characterize those points and to gain insight into the structure of $\Pc$ by discriminating between equality and strict inequality in condition 2(b).

\begin{corollary} \label{cor:MOP_active} 
	Let $x^0 \neq 0$ be a Pareto critical point such that the inequality in 2(b) of Theorem \ref{th:KKT_cond} is strict. Let $j_1 < ... < j_{n^A(x^0)}$ be the (sorted) active indices at $x^0$. Define 
	the following two mappings that lift the active components to the $n$-dimensional point $x$ and vice versa:
	\begin{align*}
	&h^1 : \R^{n^A(x^0)} \rightarrow \R^n, \quad h^1_j(\bar{x}) = 
	\begin{cases}
	0 & \text{if } j \notin A(x^0), \\
	\bar{x}_i & \text{if } j_i \in A(x^0),
	\end{cases} \\
	&h^2 : \R^n \rightarrow \R^{n^A(x^0)}, \quad h^2_i(x) = x_{j_i} \quad \forall i \in \{1,...,n^A(x^0)\}.
	\end{align*}
	Then there is an open set $U \subseteq \R^n$ with $x^0 \in U$ such that $x \in U$ being Pareto critical for \eqref{eq:MOP_l1} is equivalent to $A(x) = A(x^0)$ and $h^2(x)$ being Pareto critical for the MOP
	\begin{align} \label{eq:MOP_l1_active}
	\min_{\bar{x} \in \R^{n^A(x^0)}} \left( \begin{array}{c}
	f(h^1(\bar{x}))  \\
	\norm{\bar{x}}_1
	\end{array}\right).
	\end{align}
\end{corollary}

\begin{proof}
See Section B in the supplementary material.
\end{proof}

The previous corollary implies that if the inequality in 2(b) of Theorem \ref{th:KKT_cond} is strict for a point $x^0 \in \Pc$, then locally around $x^0$, the Pareto critical set of \eqref{eq:MOP_l1} coincides with the Pareto critical set of the MOP \eqref{eq:MOP_l1_active}, embedded from $\R^{n^A(x^0)}$ to $\R^n$ via $h^1$. The objectives of \eqref{eq:MOP_l1_active} are smooth around $h^2(x^0)$ (since all entries of $h^2(x^0)$ are non-zero by construction), such that we can apply the theory for smooth MOPs to see that the tangent space of the Pareto critical set $\Pc$ must exist in $x^0$. Thus, kinks in $\Pc$ can only arise in points where equality holds in 2(b) for some $j \notin A(x^0)$. These are precisely the points where the active set can potentially change, i.e., where an inactive index might be activated (or vice versa).

Corollary \ref{cor:MOP_active} shows that it is sufficient to consider the MOP \eqref{eq:MOP_l1_active} if we are in a point where the inequality in 2(b) of Theorem \ref{th:KKT_cond} is strict. Based on \eqref{eq:MOP_l1_active}, this allows us to construct the predictor (Section~\ref{subsec:predictor}) and corrector (Section~\ref{subsec:corrector}) as in the smooth case except for kinks.

\subsection{Predictor}
\label{subsec:predictor}

In the predictor step, we perform a step along the tangent space of the Pareto critical set from the starting point $x^0 \in \Pc$. 
As discussed above, we assume that an open set around $x^0$ exists in which the zero structure of Pareto critical points does not change, such that the original MOP locally reduces to the problem \eqref{eq:MOP_l1_active}, where both objective functions are at least twice continuously differentiable since the $\ell_1$-norm is applied only to non-zero entries.
This means that we can compute the predictor according to classical continuation methods, cf.~\cite{Hillermeier2001}.
More precisely, let $H : \R^{n^A(x^0) + 2} \rightarrow \R^{n^A(x^0)+1}$,
\begin{align*}
H(\bar{x},\alpha) = 
\begin{pmatrix}
\alpha_1 \nabla (f \circ h^1)(\bar{x}) + \alpha_2 \nabla f_{\ell_1}(\bar{x}) \\
\alpha_1 + \alpha_2 - 1 
\end{pmatrix}
\end{align*}
be the map containing the KKT conditions of \eqref{eq:MOP_l1_active}. Then the tangent space of the Pareto critical set is given by the first $n^A(x^0)$ components of
\begin{align} \label{eq:predictor_ker}
&\ker(H'(h^{2}(x^0),\alpha))\nonumber\\
= &\ker
\Big(\begin{array}{c}                           
\alpha_1 \nabla^2 (f \circ h^1)(h^2(x^0)) \\
0 \\
\end{array} \nonumber \\
&\qquad\qquad\qquad\qquad\begin{array}{cc}                           
\nabla (f \circ h^1)(h^2(x^0)) & \nabla f_{\ell_1}(h^2(x^0))\\
  1 & 1 \\
\end{array}\Big) \nonumber \\
= &\ker
\left(  \begin{matrix}                           
\alpha_1\nabla^2 f(x^0)_{|A(x^0)} & \nabla f(x^0)_{|A(x^0)} & s \\
0              & 1                       & 1 \\
\end{matrix}\right),
\end{align}
where $(\nabla^2 f(x))_{|A(x^0)}$ and $\nabla f(x)_{|A(x^0)}$ denote the Hessian and the gradient of $f$ reduced to the active indices, respectively, and $s = \sgn(h^2(x^0)) = \sgn(x^0)_{|A(x^0)}$ is the result of the componentwise sign function in the active indices of $x^0$. Note that Theorem \ref{th:KKT_cond} implies $\nabla f(x^0)_{|A(x^0)} = - \abs{(\nabla f(x^0))_a} s $.

If the reduced Hessian $\nabla^2 f(x)_{|A(x^0)}$ is regular, then the kernel (i.e., the tangent space) is one dimensional and we obtain a unique vector (up to scaling and sign). As discussed in \cite{Allgower1990}, we expect this to be the generic case and assume that this holds for the remainder of this paper. Otherwise, the kernel could be of higher dimension which would lead to infinitely many possible directions. In this case, we could try to find a suitable direction in the kernel by using the directional derivative of the gradient and the fact that the absolute values of the entries of the gradient in the active indices must coincide (cf.~2(a) of Theorem~\ref{th:KKT_cond}). 

If the reduced Hessian $\nabla^2 f(x)_{|A(x^0)}$ is regular, then a kernel vector $\bar{v}$ of \eqref{eq:predictor_ker} is given by
\begin{equation}  \tag{P} \label{eq:predictor_ker_v}
\begin{aligned}
\bar{v} &= \gamma \cdot (\bar{v}_1,\bar{v}_2,\bar{v}_3)^\top \quad \text{with} \quad \gamma \in \R \quad \text{and}\\
\bar{v}_1 &= \frac{(1 + \abs{(\nabla f(x^0))_a})}{\alpha_1} \left(\nabla^2 f(x^0)_{|A(x^0)}\right)^{-1} s,\\
&= (1 + \abs{(\nabla f(x^0))_a})^2 \left(\nabla^2 f(x^0)_{|A(x^0)}\right)^{-1} s,\\
\bar{v}_2 &= 1,\qquad
\bar{v}_3 = -1.
\end{aligned}
\end{equation}
The predictor of \eqref{eq:MOP_l1_active} is given by $\bar{v}_1$, while $\bar{v}_2$ and $\bar{v}_3$ represent the corresponding direction in the space of KKT multipliers.
The predictor of \eqref{eq:MOP_l1} is then $v_1 = h^1(\bar{v}_1)$.

\subsubsection{Direction}
If the reduced Hessian is regular, there are only two possible choices for the direction of the predictor step, namely $v_1$ and $-v_1$. As we want to move along the Pareto front in the direction where $f$ becomes smaller (and the $\ell_1$-norm larger), it would be intuitive to choose the sign of $\nabla f(x)^\top v_1$ as an indicator, as proposed in \cite{Hillermeier2001} and \cite{ParetoTracer}. 
While this mostly works, the need to cross a ``turning point'' may arise. These are points where we follow the Pareto critical set in the same direction but the corresponding direction in the objective space vanishes and then turns around. This phenomenon is not related to the non-smoothness of the $\ell_1$-norm but can also occur in the smooth parts, i.e., on a single \emph{activation structure} (the set of points $x\in\Pc$ with a constant active set $A(x)$), as shown in the following example.

\begin{example}\label{ex:turningPoint}
	Consider \eqref{eq:MOP_l1} with
	\begin{align}\label{eq:example1}
	f(x) = (x_1 + 1)^2 + (x_2 - 1)^4 - \tfrac{1}{2} \left( x_2 - \tfrac{1}{4} \right)^3.
	\end{align}
	Figure~\ref{fig:example1} shows the Pareto critical set for this problem, which can be computed by hand in this case. We see that we have two turning points $x^1$ and $x^2$ in the Pareto critical set, where the direction in the objective space turns around. 
\end{example}

\begin{figure}[h!]
	\centering
	\parbox[b]{0.24\textwidth}{\centering \includegraphics[width=0.24\textwidth]{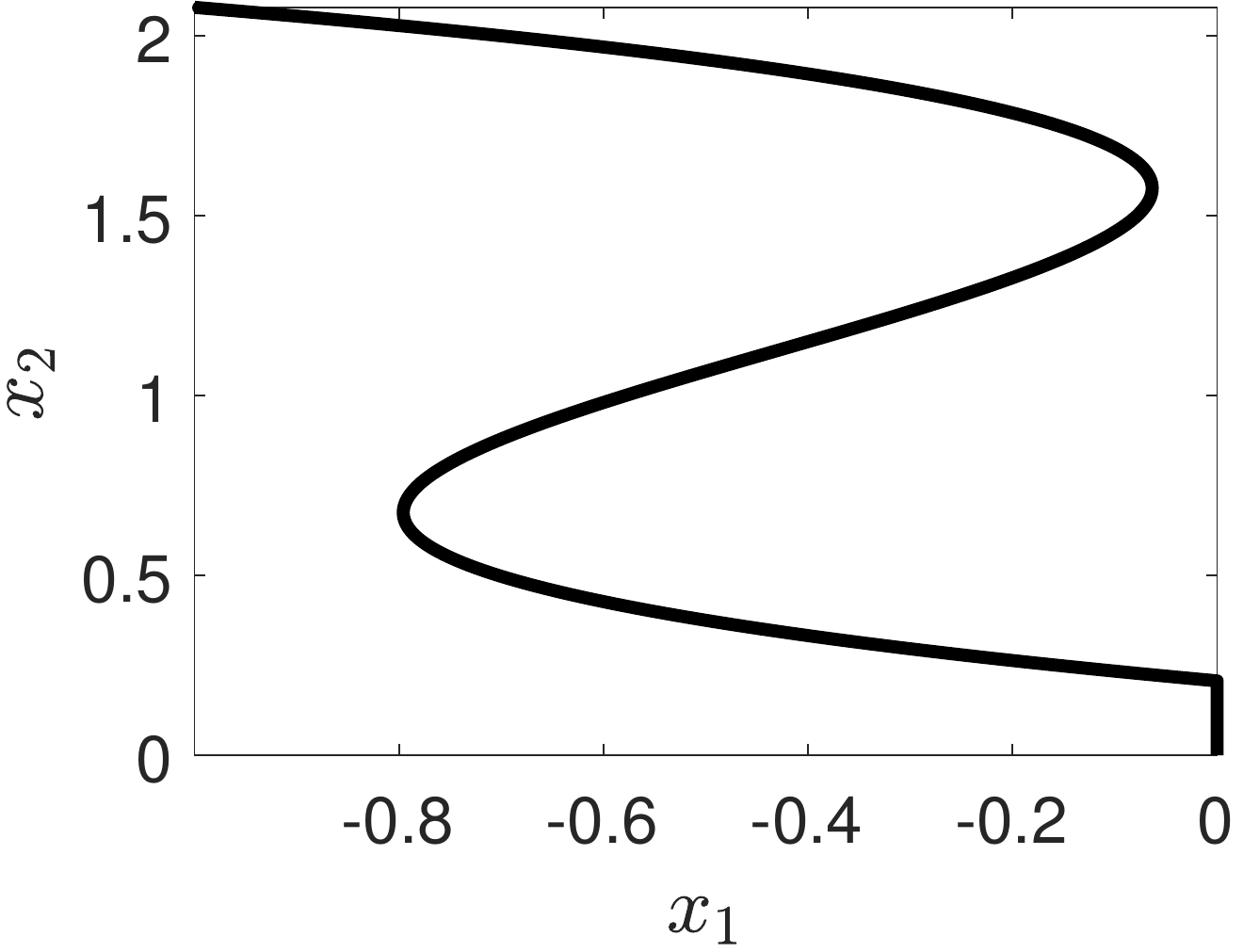} \\ (a) } \hfil
	\parbox[b]{0.24\textwidth}{\centering \includegraphics[width=0.24\textwidth]{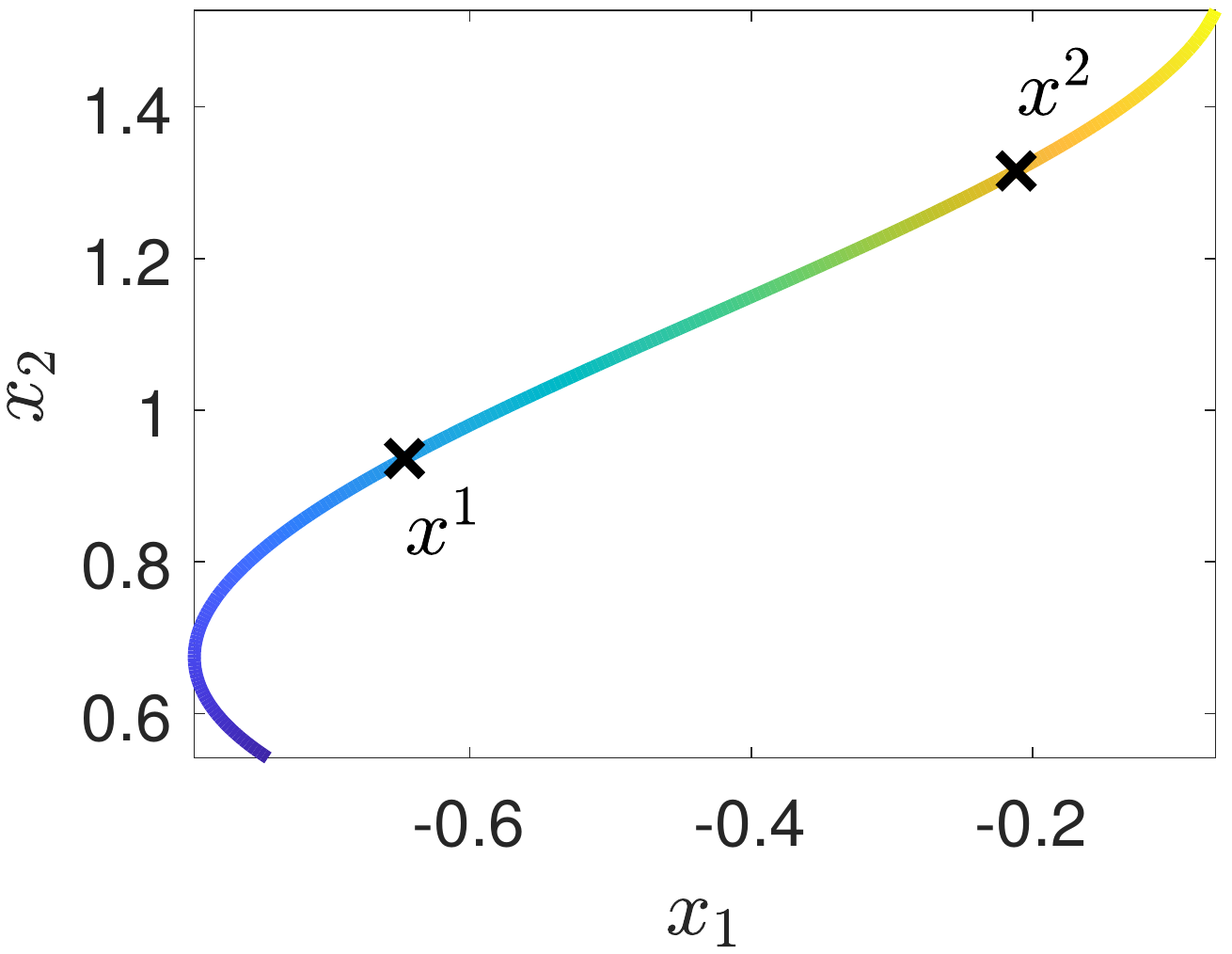} \\ (b) }
	\ \\
	\parbox[b]{0.24\textwidth}{\centering \includegraphics[width=0.24\textwidth]{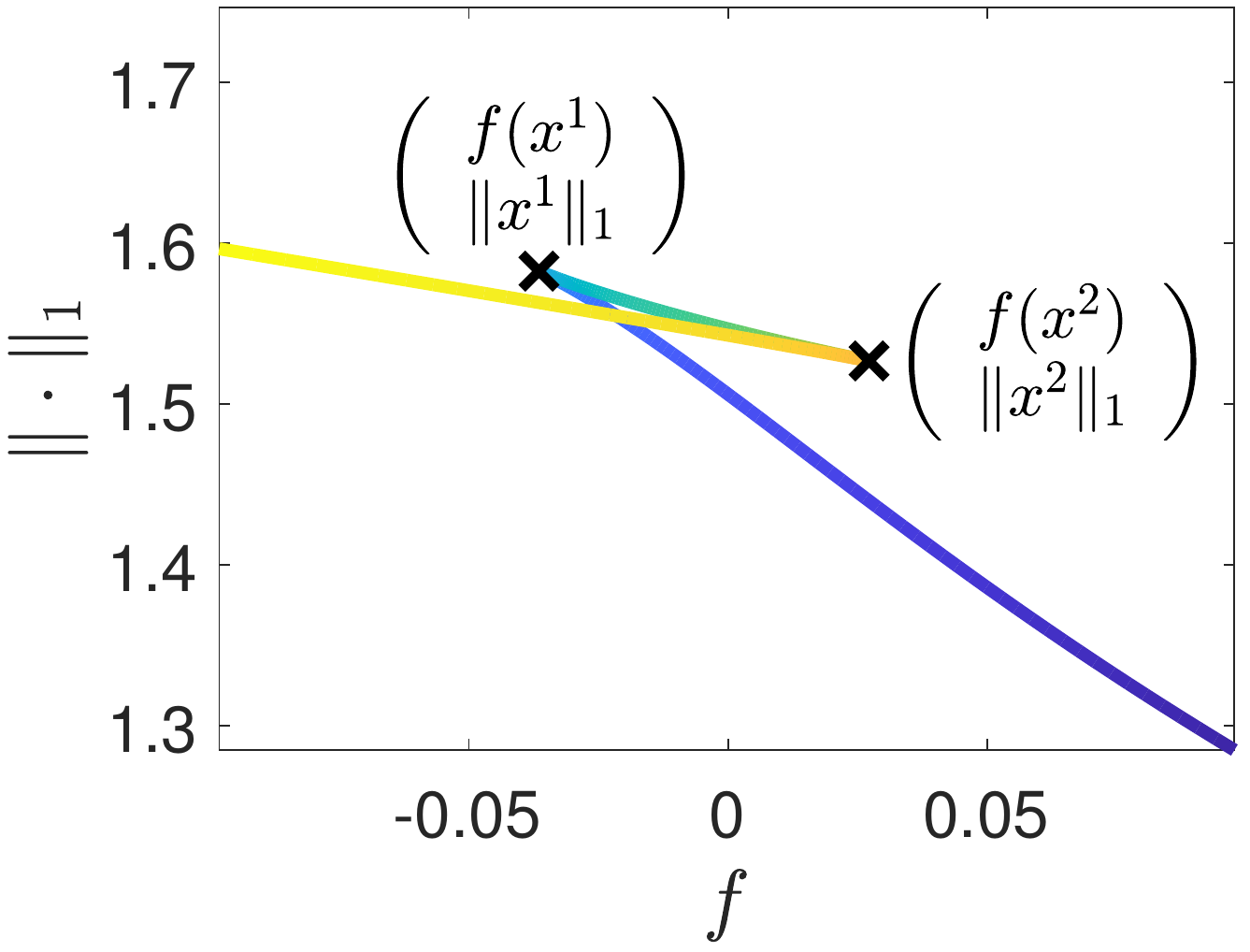} \\ (c) } \hfil
	\parbox[b]{0.24\textwidth}{\centering \includegraphics[width=0.24\textwidth]{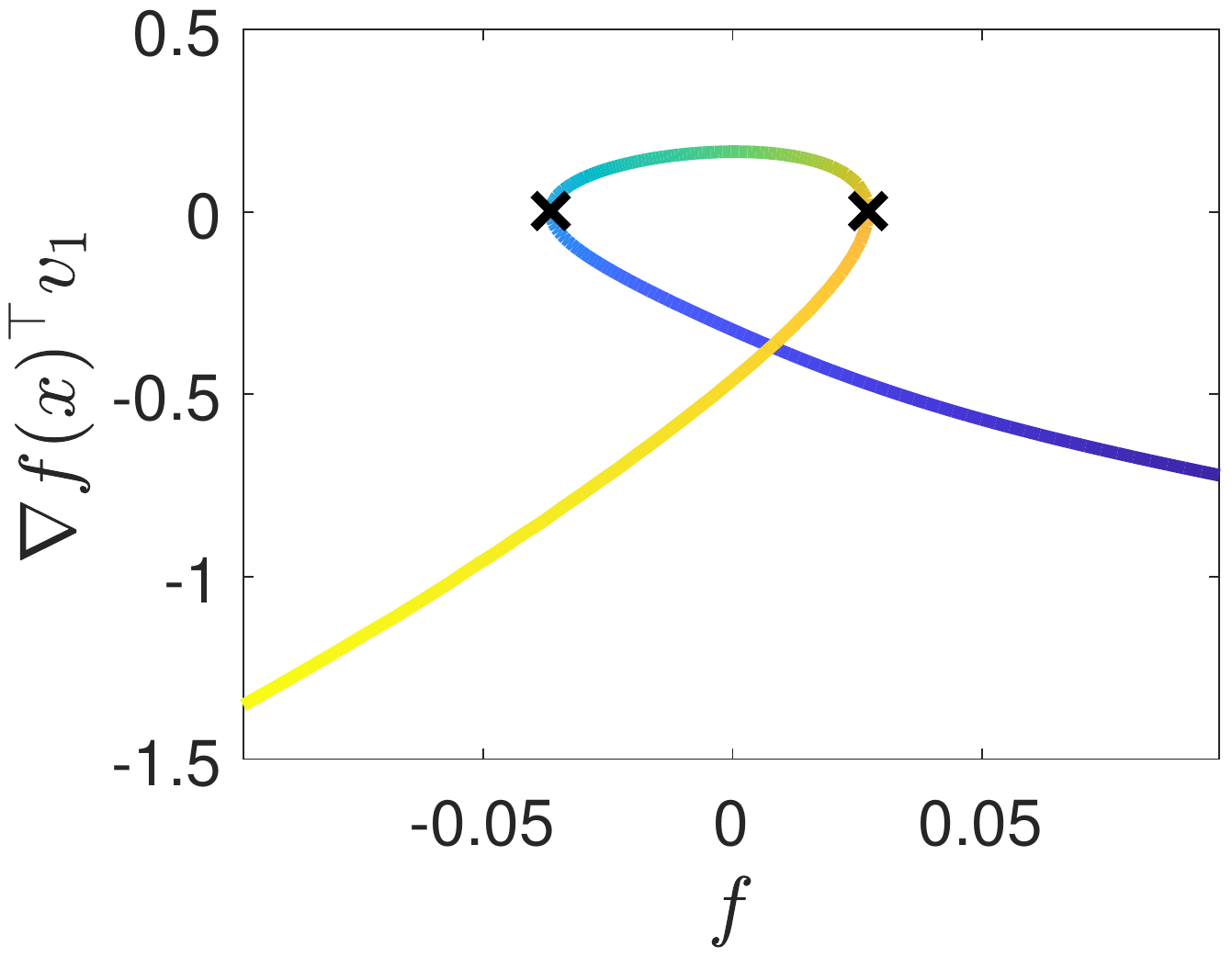} \\ (d)}
	\caption{(a) $\Pc$ for problem \eqref{eq:example1}. (b)  Part of $\Pc$ where the turning points $x^1$ and $x^2$ occur. (c) Image of $\Pc$ with the same coloring as in (b). (d) Derivative of $f$ in the direction of the predictor, i.e., $\nabla f(x)^\top v_1$.}
	\label{fig:example1}
\end{figure}

We thus use a criterion which stems from classical continuation for dynamical systems in \cite{Allgower1990}. It is based on the sign of the determinant of the augmented jacobian $\begin{pmatrix} H'(h^{2}(x^0),\alpha)^\top, \bar{v} \end{pmatrix}^\top$, which is here given by
\begin{align}
J_a := \left(  \begin{matrix}                          
\alpha_1 \nabla^2 f(x^0)_{|A(x^0)} & \nabla f(x^0)_{|A(x^0)} & s \\                           0              & 1             & 1 \\
\bar{v}_1^\top & \bar{v}_2 & \bar{v}_3\\
\end{matrix}\right).
\end{align}
If the sign of $\det{(J_a)}$ is constant in one activation structure, then this ensures that $\bar{v}$ has the same orientation with respect to the image of $H'(h^2(x^0),\alpha)$. Furthermore, it prevents us from walking back in the direction we came from (cf.~\cite{Allgower1990}).
As we will see in Section~\ref{subsec:PC_act_deact}, the correct direction is known at the beginning of a new activation structure.
Therefore, this is a suitable criterion to choose the direction in each predictor step. 
Furthermore, a simple calculation shows that
\begin{equation}\label{eq:signPredictor}
\begin{aligned}
\det{(J_a)} &= -\bar{v}_2 \cdot z \cdot \det{(\alpha_1 \nabla^2 f(x^0)_{|A(x^0)})}, 
\end{aligned}
\end{equation}
where $z :=  2 + (1+\abs{(\nabla f(x^0))_a})^4 \norm{(\nabla^2 f(x^0)_{|A(x^0)})^{-1}s}_2^2 > 0$ with $a \in A(x^0)$. As a result, only the determinant of $\nabla^2 f(x^0)_{|A(x^0)}$ has to be computed in order to determine the sign of $\det{(J_a)}$.

\subsubsection{Step size}
Now it remains to choose a step size $h \in \R^{>0}$ that determines how far we move along the tangent direction, i.e., that determines the predicted point $x^p = x^0 + h \cdot (\pm v_1)$. 
To obtain an even coverage of the Pareto critical set, we choose a constant step size in the parameter space, i.e.,
\begin{align}\label{eq:predictorStepSize}
h := \frac{\tau}{\norm{v_1}_{2}}
\end{align}
with a constant $\tau \in \R^{>0}$.
For an even distribution in the objective space, see \cite{Hillermeier2001, ParetoTracer}.

Since we still assume that we are in a part of the Pareto critical set where the activation structure and sign of the variables does not change, the issue may occur that $h$ is too large if we choose it according to one of the above formulas. Therefore, we will choose $\hat{h} = \min(h, \bar{h})$,
where $\bar{h}$ determines the maximal length of $v_1$, such that the activation structure and sign does not change, i.e., such that no active index gets inactive or crosses $0$.

\subsection{Corrector}
\label{subsec:corrector}

The corrector step maps the predicted point onto the Pareto critical set by finding the nearest point $x^c \in \Pc$ with the same zero structure and signs as the predicted point $x^p$, i.e., with $A(x^c) = A(x^p) = A(x^0)$. The standard approach is to solve the system of equations given by the KKT conditions~\eqref{eq:KKT_l1} with $x^p$ as starting point. Since the subgradient of the $\ell_1$-norm consists of $M(x^0) = 2^{n^0(x^0)}$ vectors, this can result in very large systems. Fortunately, we can exploit Theorem~\ref{th:KKT_cond} to obtain a simpler way,
and solve the following optimization problem:
\begin{equation}
\tag{C}
\begin{aligned}
\min_{x \in \R^n} &\norm{x - x^p}_2^2, \quad s.t. \\
&(\nabla f(x))_j^2 - (\nabla f(x))_a^2 = 0 &&\text{ if } x^0_j \neq 0\\
&(\nabla f(x))_j^2 - (\nabla f(x))_a^2 \le 0 &&\text{ if } x^0_j = 0\\
&(\nabla f(x))_j \le 0 &&\text{ if } x^0_j > 0\\
&(\nabla f(x))_j \ge 0 &&\text{ if } x^0_j < 0\\
&x_j \ge 0 &&\text{ if } x^0_j > 0\\
&x_j \le 0 &&\text{ if } x^0_j < 0\\
&x_j = 0 &&\text{ if } x^0_j = 0
\end{aligned}
\label{eq:corrector}
\end{equation}
where $a \in A(x^0)$ is some index that is currently active.

In practice, we use an SQP method \cite{Nocedal2006} for the solution of \eqref{eq:corrector} with $x^p$ as the initial point.
In the case of convergence issues, one can choose a starting point closer to the Pareto critical point $x^0$, i.e., $x^p + \lambda (x^0 - x^p)$ for some $\lambda \in (0,1)$.

\subsection{Changing the activation structure}
\label{subsec:PC_act_deact}

So far, we have seen how the predictor and corrector can be computed when the zero structure is locally constant around $x^0$. As discussed at the beginning of this section, this allows us to compute the smooth parts of the Pareto critical set in between the kinks. To compute the entire Pareto critical set (or a connected component of it), a mechanism that is able to change the activation structure $A(x)$ is required. 

According to Theorem~\ref{th:KKT_cond}, an index $j$ can only change from inactive to active (or vice versa) if $x_j= 0$ and $\abs{\nabla f(x)_j} = \abs{\nabla f(x)_a}$ with $a \in A(x)$. We will call such an index $j$ \emph{potentially active} and denote the set of all potentially active indices by 
\begin{equation}\label{eq:defAp}
\begin{aligned}
    A_p(x) = \{j \in \{1,\dots,n\} ~|~ &x_j = 0 \mbox{~and}\\ &\abs{\nabla f(x)_j} = \abs{\nabla f(x)_a}\}.
\end{aligned}
\end{equation}

\subsubsection{Activate indices}
\label{subsubsec:activate}
Let $x^*$ be a Pareto critical point with $A_p(x^*) \neq \emptyset$. If $x^*$ is an end point of the current activation structure, then the construction of the corrector ensures that we reach it after a finite number of steps (if the step length $\tau$ is sufficiently small). If $x^*$ is not an end point, then we will generally not be able to detect it during continuation, as we move with a constant step size. Fortunately, the latter points are less likely to exist, as will be discussed briefly at the end of this section. 

Note that not all potentially active indices necessarily need to be activated, such that a way to identify the relevant indices is required.
Denoting the number of potentially active indices with $p$, we have $2^p$ possible activation structures. 
For every possible activation structure $A^l_p \subset A_p(x^*)$, $l = 1,\dots,2^p$, we can compute a vector $\bar{v}^l = (\bar{v}_1^l,\bar{v}_2^l,\bar{v}_3^l)$ as in the predictor step (Section~\ref{subsec:predictor}). If we assume $\abs{\nabla f(x^*)_a} > 0 $, then according to 2(a) in Theorem \ref{th:KKT_cond}, the correct sign of the possibly active indices is given by $-\sgn(\nabla f(x^*))_j$. Therefore, $\bar{v}^l$ is given by

\begin{equation}
\ker\left(  \begin{matrix}                           
\alpha_1\nabla^2 f(x^*)_{|A(x^*) \cup A^l_{p}} & \nabla f(x^*)_{|A(x^*) \cup A^l_{p}} & s^l \\                           0              & 1             & 1 \\
\end{matrix}\right), \label{eq:predictor_ker_PA}
\end{equation}
with $s^l = -\sgn(\nabla f(x^*))_{|A(x^*) \cup A^l_{p}}$.
If the Hessian $\nabla^2 f(x^*)_{|A(x^*) \cup A^l_{p}}$ is regular, we obtain
\begin{equation}\label{eq:predictor_v_PA}
\begin{aligned}
\bar{v}_1^l &= (1 + \abs{(\nabla f(x^*))_a})^2 \left(\nabla^2 f(x^*)_{|A(x^*) \cup A^l_{p}}\right)^{-1} s^l,\\
\bar{v}_2^l &= 1 , \\
\bar{v}_3^l &= -1.
\end{aligned}
\end{equation}
Analogously to Section~\ref{subsec:predictor}, we obtain the direction $v_1^l \in  \R^n$ by the corresponding embedding $h^1:\R^{n^A(x^*)+\abs{A_p^l}} \rightarrow \R^n$.
If $A^l_p$ is a set of potentially active indices that can actually be activated, i.e., if there is a curve $\gamma$ of Pareto critical points starting in $x^*$ with the active set $A(x^*) \cup A^l_p$, then by construction, $v_1^l$ is the tangent vector of that curve. By analyzing the relationship between the derivative of $\gamma$ at $x^*$ and the conditions in Theorem \ref{th:KKT_cond}, we obtain necessary conditions for $A^l_p$ to be activated. This is done in the following lemma and the subsequent corollary.

\begin{lemma}\label{lemma:PA}
	Let $x^* \in P_c \setminus \{ 0 \}$ with $\nabla f(x^*) \neq 0$. Let $\gamma : (-1,1) \rightarrow \R^n$ such that
	\begin{itemize}
		\item[(i)] $\gamma$ is continuously differentiable,
		\item[(ii)] $\gamma(t) \in P_c \quad \forall t \in [0,1)$, 
		\item[(iii)] $\gamma(0) = x^*$,
		\item[(iv)] $A(\gamma(t)) =: A_\gamma$ is constant $\forall t \in (0,1)$,
	\end{itemize}
	Let $a \in A(x^*)$ and $w = \gamma'(0)$. Then for all $j \in A_\gamma \setminus A(x^*)$ we have
	\begin{enumerate}[label=(\alph*)]
		\item either $w_j = 0$ or $\sgn(w_j) = - \sgn((\nabla f(x^*))_j)$,
	\end{enumerate}
	and for all $j \in A_p(x^*)\setminus A_\gamma$ we get
	\begin{enumerate}[resume, label=(\alph*)]
		\item \quad\,$\sgn((\nabla f(x^*))_j) (\nabla^2 f(x^*))_j w$ \\ $\le \sgn((\nabla f(x^*))_a) (\nabla^2 f(x^*))_a w$,
	\end{enumerate}
	where $(\nabla^2 f(x^*))_j$ is the $j$-th row of $\nabla^2 f(x^*)$.
\end{lemma}

\begin{proof}
See Section C in the supplementary material.
\end{proof}

Unfortunately, without knowledge of the Pareto critical set, the existence of the curve $\gamma$ in the previous lemma is almost impossible to verify in practice. 
Recall that Corollary \ref{cor:MOP_active} indicates that $\Pc$ is piecewise smooth, where parts in between kinks are embeddings of Pareto critical sets of the smooth problems \eqref{eq:MOP_l1_active} into $\R^n$. 
Thus, the existence of $\gamma_{|(0,1)}$ follows from the smooth structure of \eqref{eq:MOP_l1_active}. 
In light of this, the existence of $\gamma$ implies a regularity of the smooth parts when they reach a kink. 
For example, roughly speaking, continuous differentiability of $\gamma$ on an open neighborhood of $t = 0$ implies that the ``curvature'' of $\Pc$ is bounded.

Note that the statement of the previous lemma only involves the derivative $w = \gamma'(0)$ and the new active set $A_\gamma$. 
By the construction in \eqref{eq:predictor_v_PA}, $w$ is (up to scaling) equal to the direction $v_1^l$ for the index set $A^l_p = A_\gamma \setminus A(x^*)$. 
Although $A_\gamma$ is not known in practice, we can use this relation to identify the indices that need to be activated, as shown in the following corollary.

\begin{corollary}\label{cor:PA}
	Let $x^* \in P_c \setminus \{ 0 \}$ with $\nabla f(x^*) \neq 0$, let $\gamma$ be a curve as in Lemma \ref{lemma:PA} and let $a \in A(x^*)$.\\
	Let $v^l_1$ be the vector computed via \eqref{eq:predictor_v_PA} and the appropriate embedding $h^1$ corresponding to the index set $A^l_p = A_\gamma \setminus A(x^*)$. Then, for $w = \sigma v_1^l$ with $\sigma \in \{-1,+1\}$,
	\begin{enumerate}[label=(\alph*)]
		\item  \label{item:PA_1}
		either $w_j = 0$ or $\sgn(w_j) = - \sgn((\nabla f(x^*))_j)$,
	\end{enumerate}
	and for all $j \in A_p(x^*)\setminus A_\gamma$,
	\begin{enumerate}[resume,label=(\alph*)]
		\item $\sgn((\nabla f(x^*))_j) (\nabla^2 f(x^*))_j w \le -\sigma (1+\abs{(\nabla f(x))_a})^2$,
	\end{enumerate}
	where $(\nabla^2 f(x^*))_j$ is the $j$-th row of $\nabla^2 f(x^*)$.
\end{corollary}

\begin{remark}\label{rem:PA_directions}
	According to (a) in Corollary~\ref{cor:PA}, for all possible activation structures $A_p^l(x^*) \neq \emptyset$, there either exists only one valid direction ($+v_1^l$ or $-v^l_1$) or the tangent vector $w$ of the corresponding curve $\gamma$ has to satisfy
	$w_j = 0$ for all $j \in A_p(x^*)$, i.e., $\gamma$ has to be tangential to a vector with activation structure $A(x^*)$, which is unlikely.
\end{remark}

Using the conditions from Corollary~\ref{cor:PA}, we now obtain a predictor for the suitable activation structures with the correct sign. Since we can exclude the direction we are coming from, we have to follow each of the remaining directions. This is shown in Examples~\ref{ex:PA_1} and \ref{ex:PA_2}.

\begin{example}\label{ex:PA_1}
	Solve \eqref{eq:MOP_l1} with   
	\begin{align}\label{eq:example2}
	f(x) = (x_1 - 2)^2 + (x_2 - 1)^2 + (x_3 - 1)^2.
	\end{align}
	Figure~\ref{fig:example2} shows the Pareto critical set for this problem. Consider $x^* = (1,0,0)^\top$. Since $\nabla f(x^*) = (-2,-2,-2)^\top$, we have $A_p(x^*) = \{2,3\}$ (cf.~\eqref{eq:defAp}). According to \eqref{eq:predictor_v_PA}, we get the following possible directions:
	\begin{align*}
	v_1^0 = \begin{pmatrix} \frac{9}{2} \\[0.3em] 0 \\[0.3em] 0 \end{pmatrix}, 
	v_1^1=\begin{pmatrix} \frac{9}{2} \\[0.3em] \frac{9}{2} \\[0.3em] 0 \end{pmatrix},
	v_1^2=\begin{pmatrix} \frac{9}{2} \\[0.3em] 0 \\[0.3em] \frac{9}{2} \end{pmatrix},
	v_1^3=\begin{pmatrix} \frac{9}{4} \\[0.3em] \frac{9}{4} \\[0.3em] \frac{9}{4}\end{pmatrix}
	\end{align*}
	After checking condition (a) in Corollary~\ref{cor:PA}, we know that only
	\begin{align*}
	v_1^0, -v_1^0, v_1^1, v_1^2 \text{ and } v_1^3
	\end{align*}
	remain as suitable directions. By checking condition (b), the set of admissible directions can be reduced to 
	\begin{align*}
	-v_1^0 \text{ and } v_1^3.
	\end{align*}
	Assuming that we started in the origin, we have to walk in direction $v_1^3$, i.e., both $j = 2$ and $j=3$ become active (and positive).
\end{example}

\begin{example}\label{ex:PA_2}
	Solve \eqref{eq:MOP_l1} with 
	\begin{align}\label{eq:example3}
	f(x) = (x_1 - 2)^2 + (x_2 - 1)^2 + 2 x_1 x_3.
	\end{align}
	Figure~\ref{fig:example3} shows the relevant part of $\Pc$ for this problem. Consider $x^* = (1,0,0)^\top$. Since $\nabla f(x^*) = (-2, -2, 2)^\top$, we have $A_p(x^*) = \{2,3\}$. According to \eqref{eq:predictor_v_PA}, we get as possible directions:
	\begin{align*}
	v_1^0 = \begin{pmatrix} \frac{9}{2} \\[0.3em] 0 \\[0.3em] 0 \end{pmatrix}, 
	v_1^1=\begin{pmatrix} \frac{9}{2} \\[0.3em] \frac{9}{2} \\[0.3em] 0 \end{pmatrix},
	v_1^2=\begin{pmatrix} -\frac{9}{2} \\[0.3em] 0 \\[0.3em] 9 \end{pmatrix},
	v_1^3=\begin{pmatrix} -\frac{9}{2} \\[0.3em] \frac{9}{2} \\[0.3em] 9 \end{pmatrix}
	\end{align*}
	Verification of conditions (a) and (b) in Corollary~\ref{cor:PA} reduces the set of admissible directions to $- v_1^0$, $v_1^1$  and $v_1^2$.
	Assuming that we started in the origin, we have to (try to) advance with the continuation into the directions $v_1^1$ and $-v_1^2$. A predictor step along $v_1^1$ directly terminates since the corrector delivers $x^*$ as the nearest Pareto critical point and therefore, the remaining direction is $-v_1^2$. Thus, only $j = 3$ is activated.
\end{example}

\begin{figure}[h!]
	\centering
	\begin{subfigure}{0.24\textwidth}
		\includegraphics[width=\textwidth]{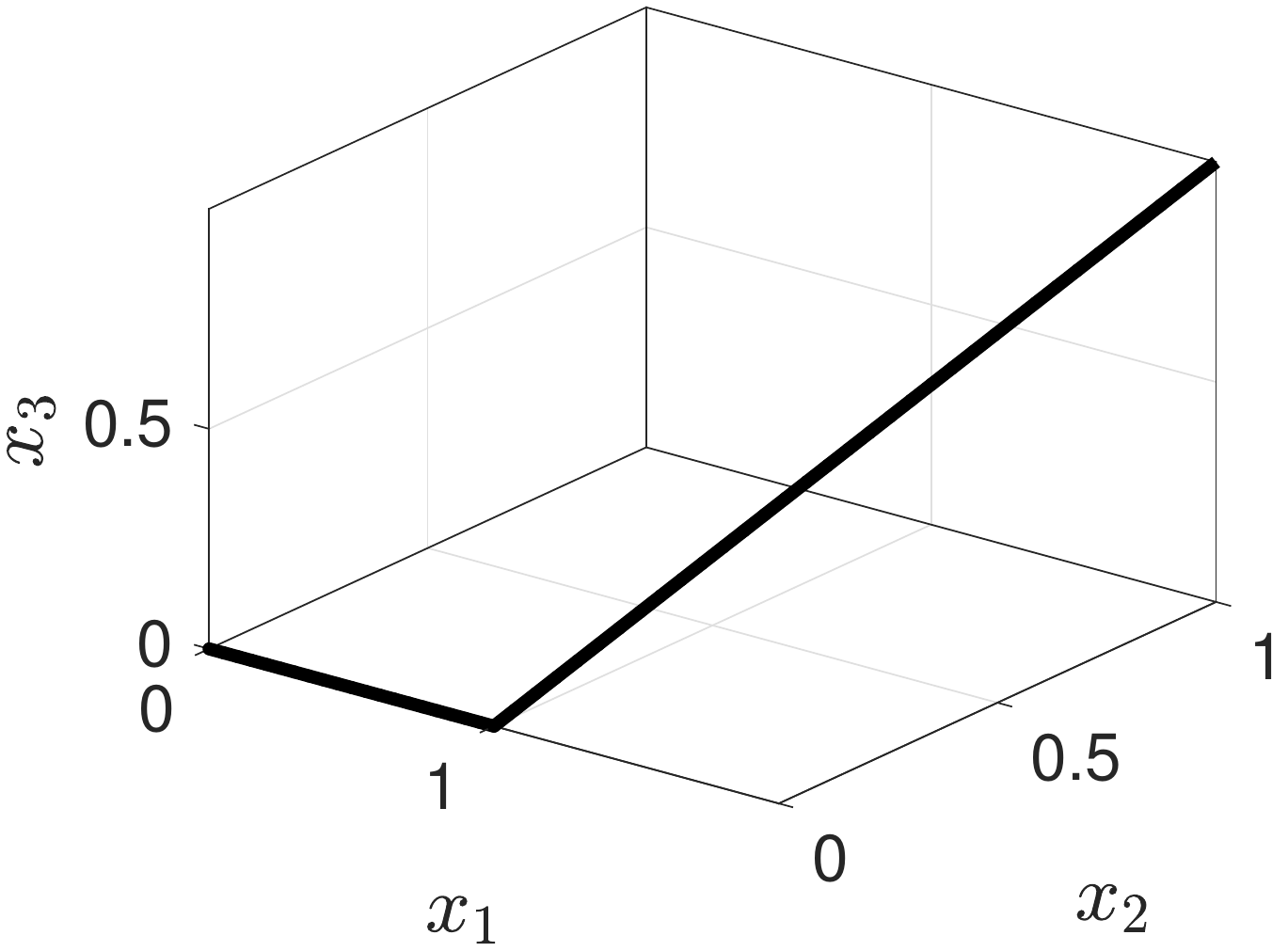}
		\caption{}\label{fig:example2}
	\end{subfigure}
	\hfill
	\begin{subfigure}{0.24\textwidth}
		\includegraphics[width=\textwidth]{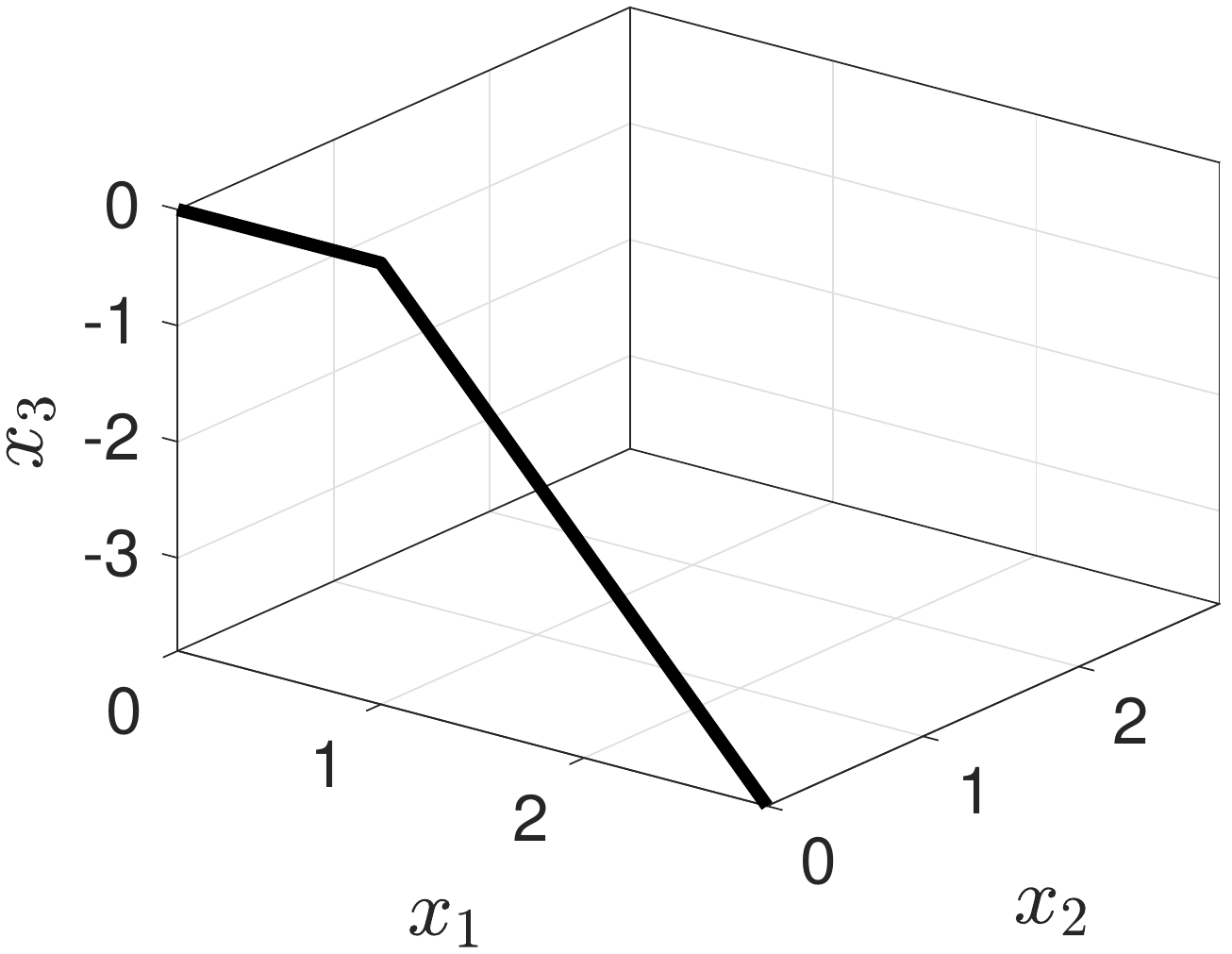}
		\caption{}\label{fig:example3}
	\end{subfigure}
	\caption{(a) $\Pc$ for problem \eqref{eq:example2}. (b)  Part of $\Pc$ of problem \eqref{eq:example3}.}
	\label{fig:example2_3}
\end{figure}

In practice, often only two possible directions remain after applying Corollary~\ref{cor:PA} -- the old and new direction. Nevertheless, it can occur that the Pareto set splits, i.e., that there are at least 3 suitable directions. 
In this case, the continuation method needs to be advanced in each direction individually. An example where this is demonstrated can be found in Section A of the supplementary material.
If no direction satisfies the conditions in Corollary~\ref{cor:PA}, we have reached the end of a connected component of $\Pc$.
In the smooth parts of the Pareto critical set with constant activation structure, a connected component can only end if $\nabla f(x) = 0$, i.e., a local minimizer for the main objective has been found \cite{GPD2019}.

\begin{remark}\label{rem:PA_grad_zero}
	Although the requirements of Corollary~\ref{cor:PA} do not hold if $\nabla f(x) = 0$, similar conditions can be obtained for this case. First, note that $\nabla f(x) = 0$ implies that $A_p(x^*)$ is the set of all inactive indices. Since we cannot predict the sign of $x_j$ with $j \in A_p(x^*)$ via the sign of the gradient, we have to consider every possible combination of the signs of potentially active indices for every possible activation structure $A_p^l$. The conditions of Corollary~\ref{cor:PA} can then be used as necessary conditions, where the sign of the gradient is replaced by the opposite sign of the currently considered activation and sign structure. 
	An additional condition can be obtained by using the fact that $\sgn(\nabla f(x)_j) = - \sgn(x_j)$ for all $j \in A(x)$ has to hold. 
	If the number of inactive indices is large, this results in a very large number of directions that have to be checked. 
	Therefore a more efficient approach needs to be developed in future work for very high-dimensional problems.
\end{remark}

Finally, we note that it is in theory possible that indices are activated in points besides the end points of activation structures. In other words, the Pareto critical set may split into two branches, with one branch continuing along the old structure. Since we move through smooth parts of the Pareto critical set with discrete steps, points like this cannot be detected by our method. But based on Corollary~\ref{cor:PA} and Remark~\ref{rem:PA_directions}, one can show that such points are unlikely to exist in practice.

\par 

\subsubsection{Deactivate indices}
A situation where indices have to be deactivated is when the predictor crosses zero for some entry, i.e., $\sgn(x^p_j) \neq \sgn(x^0_j)$ for some $j \in A(x^0)$.
In this situation, we reduce the predictor step length $h$ to $\bar{h}$ as the largest step size such that
\begin{align*}
x^p_j \leq 0 \ \text{if} \ x^0_j < 0 \quad\mbox{and}\quad
x^p_j \geq 0 \ \text{if} \ x^0_j > 0,
\end{align*}
for all $j \in A(x^0)$, and the indices in $A(x^0) \setminus A(x^p)$ are considered as inactive in the subsequent corrector step. If the corresponding corrector step \eqref{eq:corrector} has a solution, then the continuation method can continue on the new activation structure. Otherwise, we choose a step length in $(0,\bar{h})$ to find a new Pareto critical point with respect to the old active set $A(x^0)$. 
The likelihood of an active index being falsely deactivated by this mechanism depends on the constant step size $\tau$ and the curvature of the Pareto critical set.

A second scenario where active indices can be deactivated is the case where the corrector computes a point $x^c$ where one of the previously active indices is now zero, i.e., $x^c_i = 0$ for some $i \in A(x^0)$. This implies that some of the previously active indices $A(x^0)$ are now only potentially active. In this case, we proceed as in the activation procedure described in Section~\ref{subsubsec:activate}.

\subsection{Implementation \& Globalization}
\label{subsec:Implementation}

We now have all the ingredients to implement Algorithm~\ref{alg:PC_rough_concept}.\\
\textbf{Activation of indices:}
\begin{enumerate}
    \item If $A_p(x)$ (Eq.~\eqref{eq:defAp}) is non-empty, calculate  $v^l$ for all $A_p^l \subset A_p(x)$ using \eqref{eq:predictor_v_PA},
    \item Reduce number of directions using Corollary~\ref{cor:PA},
    \item If more than one potential direction remains, proceed in each direction. In most cases, the corrector will yield the information that $\Pc$ only continues in one direction.
\end{enumerate}
\textbf{Predictor step} (when $A_p(x)=\emptyset$)\textbf{:}
\begin{enumerate}
    \item Compute direction $v_1 = h^1(\bar{v}_1)$ using \eqref{eq:predictor_ker_v}, 
    \item Determine sign of $v_1$ using \eqref{eq:signPredictor},
    \item Compute the step size along $\pm v_1$ using \eqref{eq:predictorStepSize},
    \item Determine $x^p = x^0 + h \cdot (\pm v_1)$.
\end{enumerate}
\textbf{Deactivation of indices:} Set the indices $j$ to zero for which $\sgn(x^p_j)$ changes and reduce the predictor step length $h$ to $\bar{h}$ to avoid sign changes.\\
\textbf{Corrector step:} Compute point $x^c\in\Pc$ using \eqref{eq:corrector} with $x^p$ as the initial condition. \\

For the starting point, an intuitive choice would be $x_{\text{start}} = 0$, which is Pareto critical and even Pareto optimal as the global minimal point of the $\ell_1$-norm.
According to Theorem~\ref{th:KKT_cond}, the first index that has to be activated in $x_{\text{start}} = 0$ is the one which corresponds to the largest absolute value of the gradient. 
If there are more than one maximal entries, we can apply the strategy described in Section~\ref{subsec:PC_act_deact}.
Of course, every other Pareto critical point would also work as a starting point, from which we would have to walk in both directions of the predictor.

By construction, our method only computes the connected component of the Pareto critical set that contains the starting point $x_{\text{start}}$. Thus, if there are multiple connected components, then additional starting points are required to restart our method once the first component is computed. In practice, this can be done by considering the objective space: Due to the definition of Pareto optimality, the new starting point should have a smaller $f$ value and a larger $\ell_1$-norm or vice versa (depending on if we reached the lower end or the upper end of the image of the component). This defines an area in the objective space in which the image of the new point should lie. To actually find it, the $\varepsilon$\emph{-constraint method} \cite{Ehrgott2005} can be used, which is a solution method that can be restricted to certain parts of the objective space. Alternatively, \emph{reference point methods} \cite{Ehrgott2005} can be used, where the solution can be influenced by choosing an appropriate \emph{reference point} in the objective space. 

Finally, it is important to note that our method computes points that are Pareto \emph{critical}, but not necessarily Pareto \emph{optimal}.
To obtain the actual Pareto set, we can apply a \emph{non-dominance test} \cite{S2003}. Since the output of our method is a fine pointwise approximation of the Pareto critical set (which contains the Pareto set), this will generally result in a fine approximation of the Pareto set.

\section{Numerical Results}\label{sec:numerics}
We now study the behavior of the continuation method using three examples. We begin with two nonlinear toy examples, and then address the identification of governing equations from data via the SINDy algorithm \cite{BPK16}. Note that the latter is a regression problem and can consequently be adressed by existing homotopy approaches as well as the weighted sum method. Finally, the training of a neural network is studied.
While the cost of computing the entire Pareto set via continuation is certainly higher in comparison to calculating only a single solution, we note that (a) the individual corrector problems can be solved rather quickly due to good initial guesses and (b) the randomness (and thus, the necessity to perform multiple runs) can be reduced.

\subsection{Toy-Examples}
As a first example and to highlight the important aspects of nonlinear continuation, we consider \eqref{eq:MOP_l1} with 
\begin{equation*}
    f(x) = \left(x_1-\tfrac{1}{4}\right)^2 + \left(x_2 - \tfrac{1}{2}\right)^2 + \left(x_3 - 1\right)^4 - \tfrac{1}{2} \left(x_3 - \tfrac{1}{4}\right)^3.
\end{equation*}
Figure~\ref{fig:Toy_example} shows the result of our method. 
\begin{figure}[t!]
	\centering
	\parbox[b]{0.24\textwidth}{\centering \includegraphics[width=0.24\textwidth]{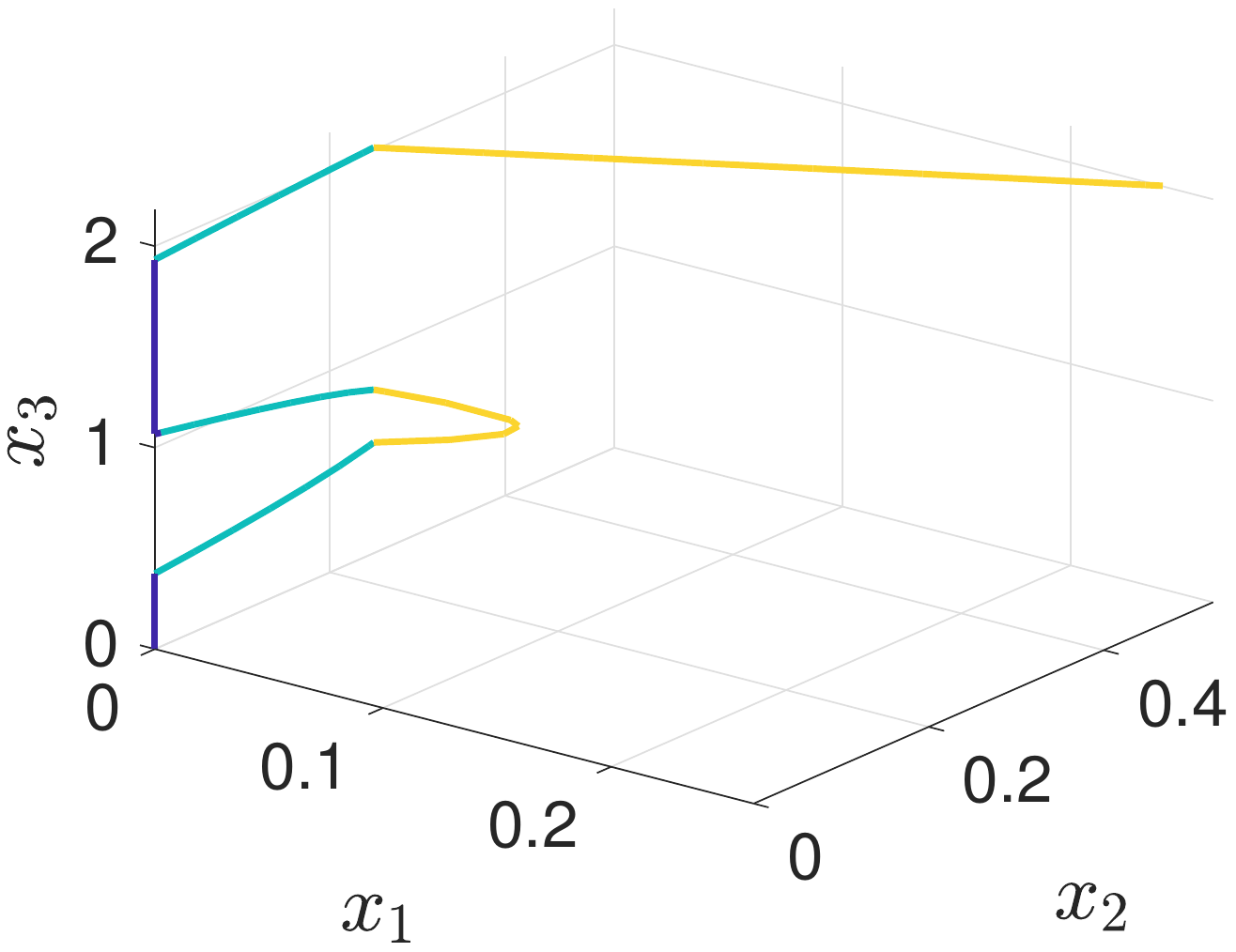} \\ (a) } \hfil
	\parbox[b]{0.24\textwidth}{\centering \includegraphics[width=0.24\textwidth]{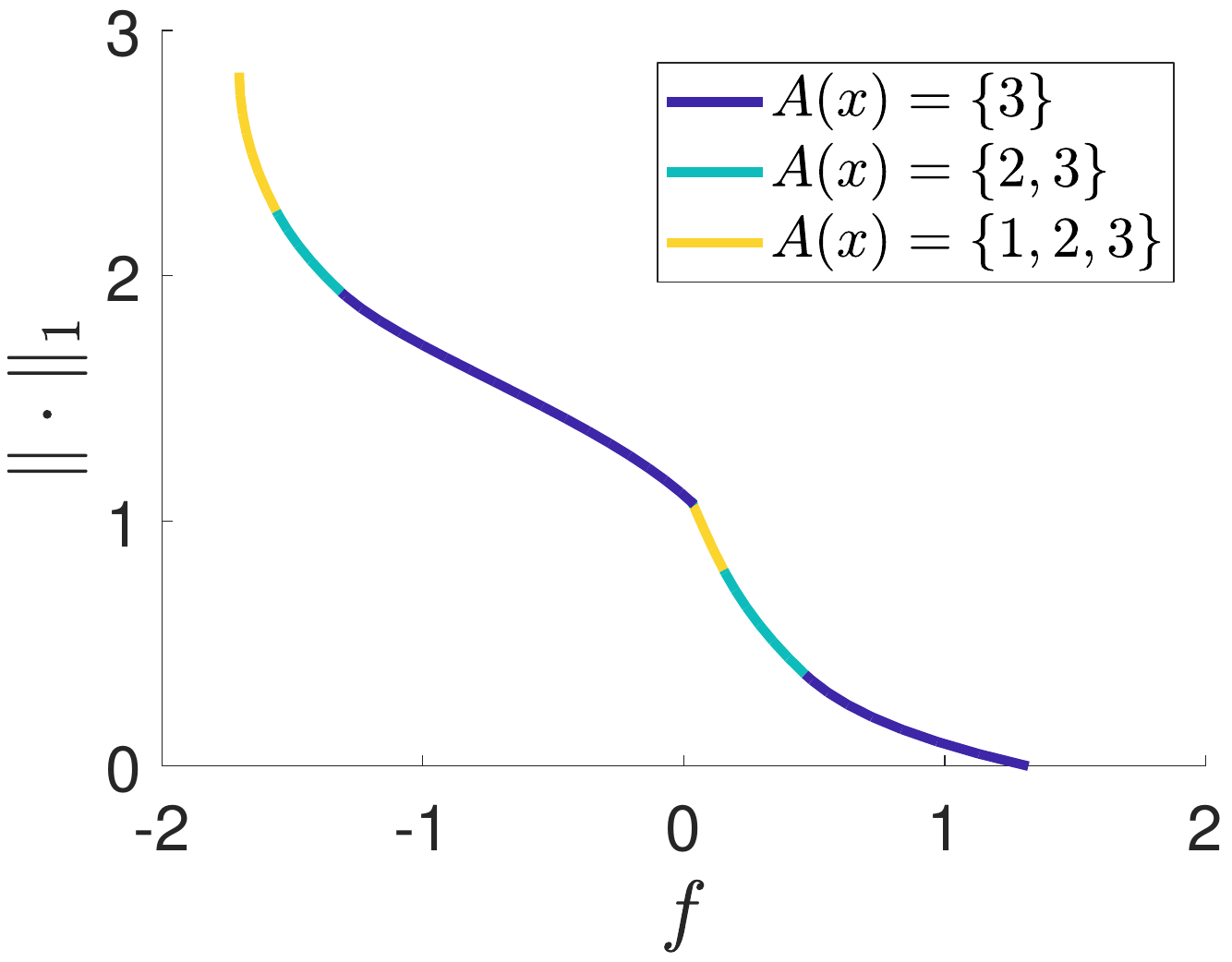} \\ (b) } \ \\
	\parbox[b]{0.24\textwidth}{\centering \includegraphics[width=0.24\textwidth]{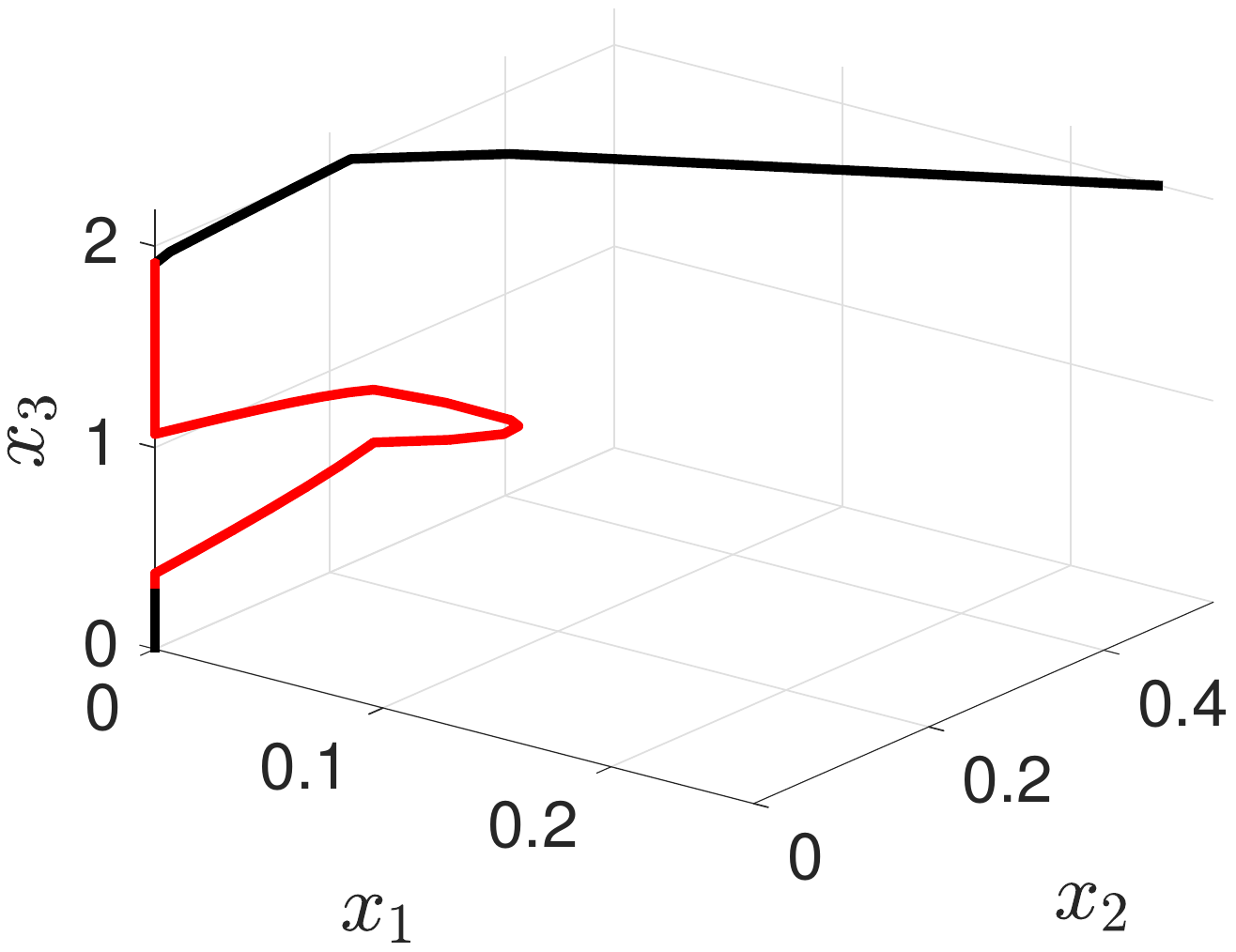} \\ (c) } \hfil
	\parbox[b]{0.24\textwidth}{\centering \includegraphics[width=0.24\textwidth]{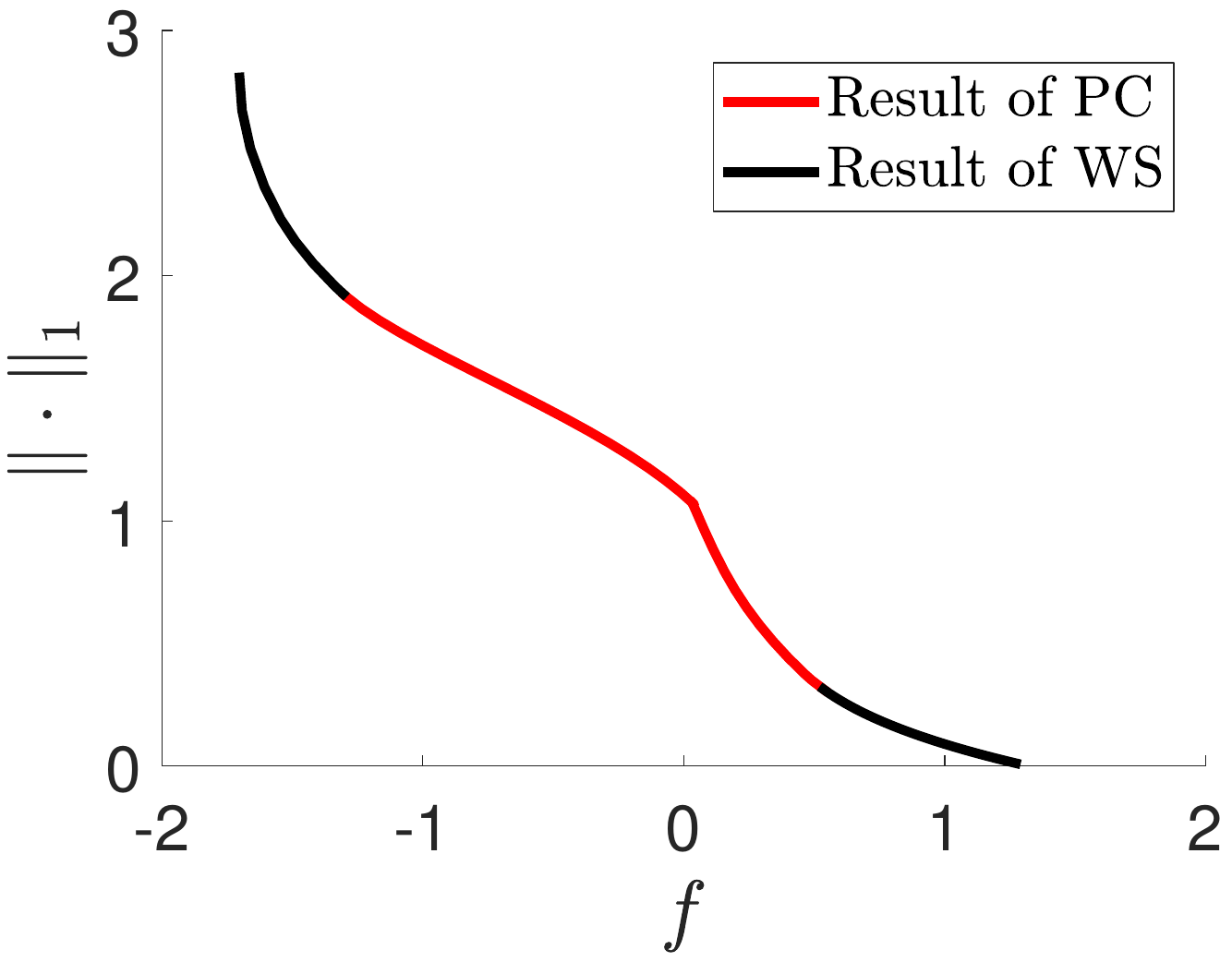} \\ (d) }
	\caption{(a) $\Pc$ for the toy example. (b) Image of $\Pc$. (c) and (d) Comparison to the weighted sum method, with which the red parts cannot be computed.}
	\label{fig:Toy_example}
\end{figure}
We observe that indices are activated and deactivated multiple times. Points where the activation structure changes are $x^1 = (0,0,0.375)$, $x^2 = (0,0.25,\frac{83}{151})$, $x^3 \approx (0,0.25,0.81)$, $x^4 \approx (0,0,1.07)$, $x^5 \approx (0,0,1.93)$ and $x^6 \approx (0,0.25,2.014)$.
In Figure~\ref{fig:Toy_example}(c) and (d), we see that only the convex part of $\Pc$ can be computed with the weighted sum method. In particular, a large part where only $j=3$ is active is missing. 

To demonstrate the feasibility in high dimensions, let
\begin{equation}\label{eq:polynom}
    f(x) = (x_1 - a_1)^4 + \sum_{i=2}^n (x_i -a_i)^2,
\end{equation}
where $a_1 = 1$ and $a_i, i =2,\dots,n$, are random values in $[-1,1]$ (similar to \cite{ParetoTracer}, Problem 4). 
Figure~\ref{fig:example_polynom} shows the result of our method for $n = 1000$ with step size $\tau=0.25$, which coincides with the analytical solution.  
\begin{figure}[h!]
    \centering
    \includegraphics[width=0.24\textwidth]{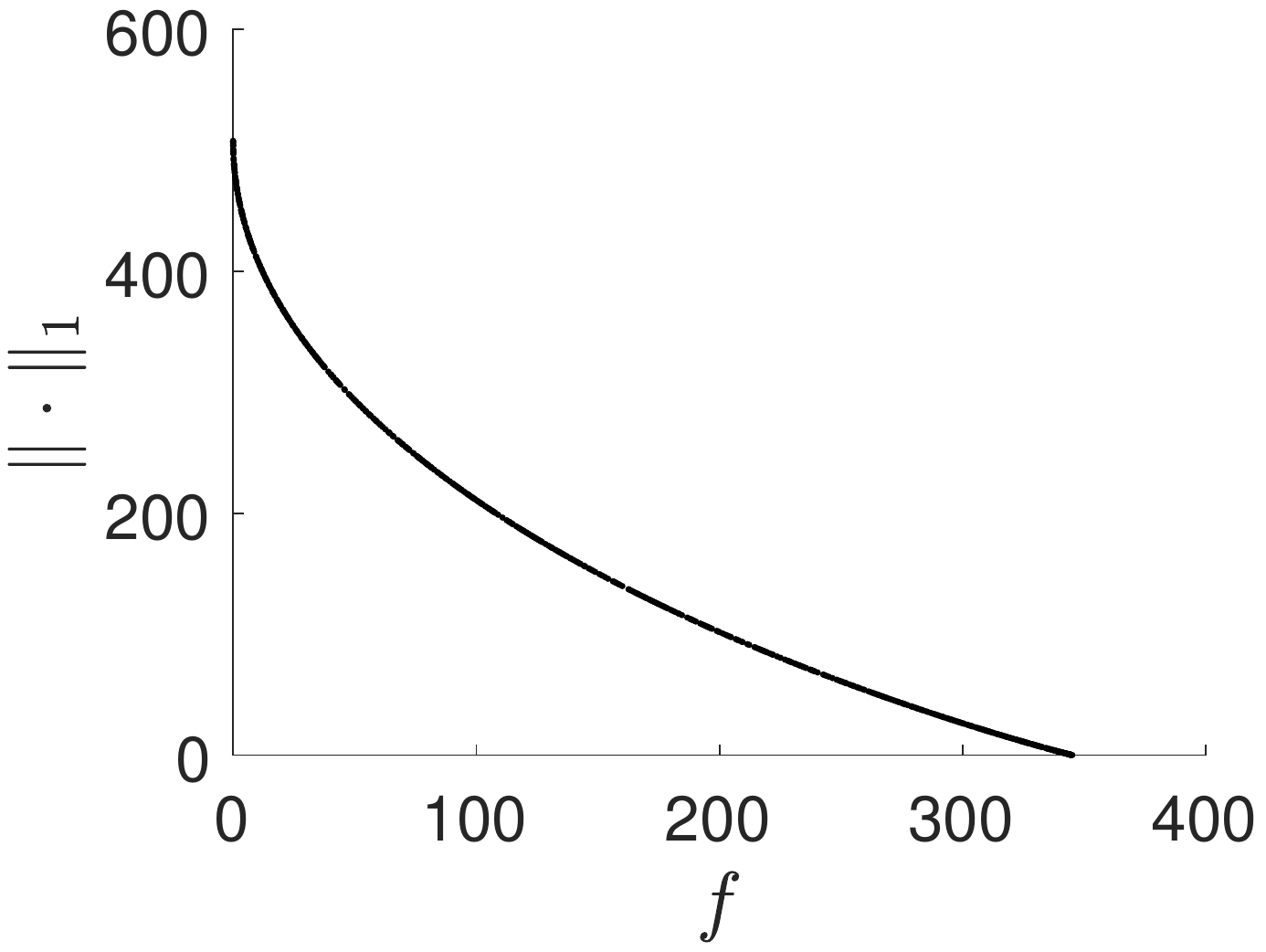}
    \caption{Pareto front for \eqref{eq:MOP_l1} with $f$ according to \eqref{eq:polynom}.}
    \label{fig:example_polynom}
\end{figure}

\subsection{SINDy}
The second example addresses the sparse identification of nonlinear dynamical systems \cite{BPK16}, where the governing system of ordinary (or partial) differential equations is sought:
\begin{equation*}
    \dot{x}(t) = g(x(t)).
\end{equation*}
Here, $x(t)\in\R^m$ is the $m$-dimensional system state at time $t$, whereas the function $g$ describing the system dynamics is unkown. The SINDy approach is then to represent $g$ via a dictionary of $c$ Ansatz functions (denoted by $\Theta(x)$) with the corresponding coefficients $W\in\R^{c\times m}$:
\begin{equation*}
    g(x(t)) = W^\top \Theta(x(t)).
\end{equation*}
The entries of $\Theta$ are arbitrary linear and nonlinear functions such as sine and cosine or polynomials. Using $N$ (potentially noisy) training data points $((x^1, \dot{x}^1),...,(x^N, \dot{x}^N))$, the objective is to identify a sparse representation of $g$ in terms of $\Theta$, and thus to obtain a robust and interpretable dynamical system from data. The corresponding main objective is thus the following least squares term:
\begin{equation*}
    f(W) =\frac{1}{mcN}\sum_{i = 1}^N \norm{W^\top \Theta(x^i) - \dot{x}^i}_2^2.
\end{equation*}
As this is a convex function, the corresponding MOP is also convex and can thus be addressed using the standard penalization approach, i.e., the weighted sum method $\min_W f(W) + \lambda \|W\|_1$ (where $W$ is transformed to one large weight vector). However, as we will see, the solution may be highly sensitive to the choice of $\lambda$.

We consider the well-known Lorenz system with parameters $\sigma = 10$, $\rho = 28$ and $\beta=\frac{8}{3}$ for which the dynamical system posses a chaotic solution:
\begin{equation*}
    \dot{x} = \begin{pmatrix}
    \sigma\cdot(x_2-x_1)\\
    x_1 \cdot (\rho-x_3)-x_2\\
     x_1 \cdot x_2 - \beta\cdot x_3
    \end{pmatrix}.
\end{equation*}
We collect $N=10,000$ training data points of a single long-term simulation to which we add white noise. The dictionary $\Theta$ consists of polynomial terms up to order $3$ and is thus comprised of $c=20$ terms, i.e., $W \in \R^{20 \times 3}$.

\begin{figure}[b!]
	\centering
	\parbox[b]{0.24\textwidth}{\centering \includegraphics[width=0.24\textwidth]{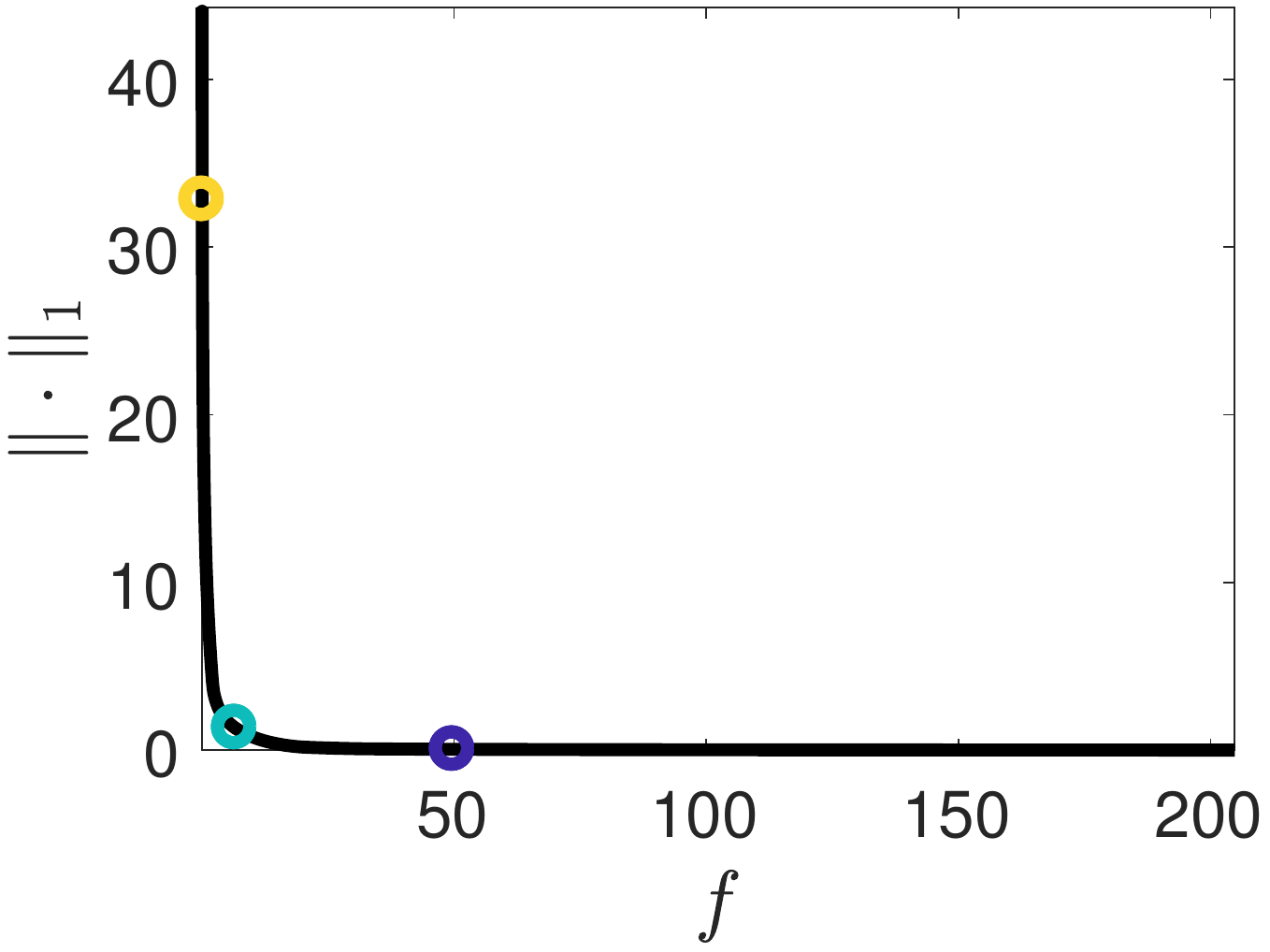} \\ (a) } \hfil
	\parbox[b]{0.24\textwidth}{\centering \includegraphics[width=0.24\textwidth]{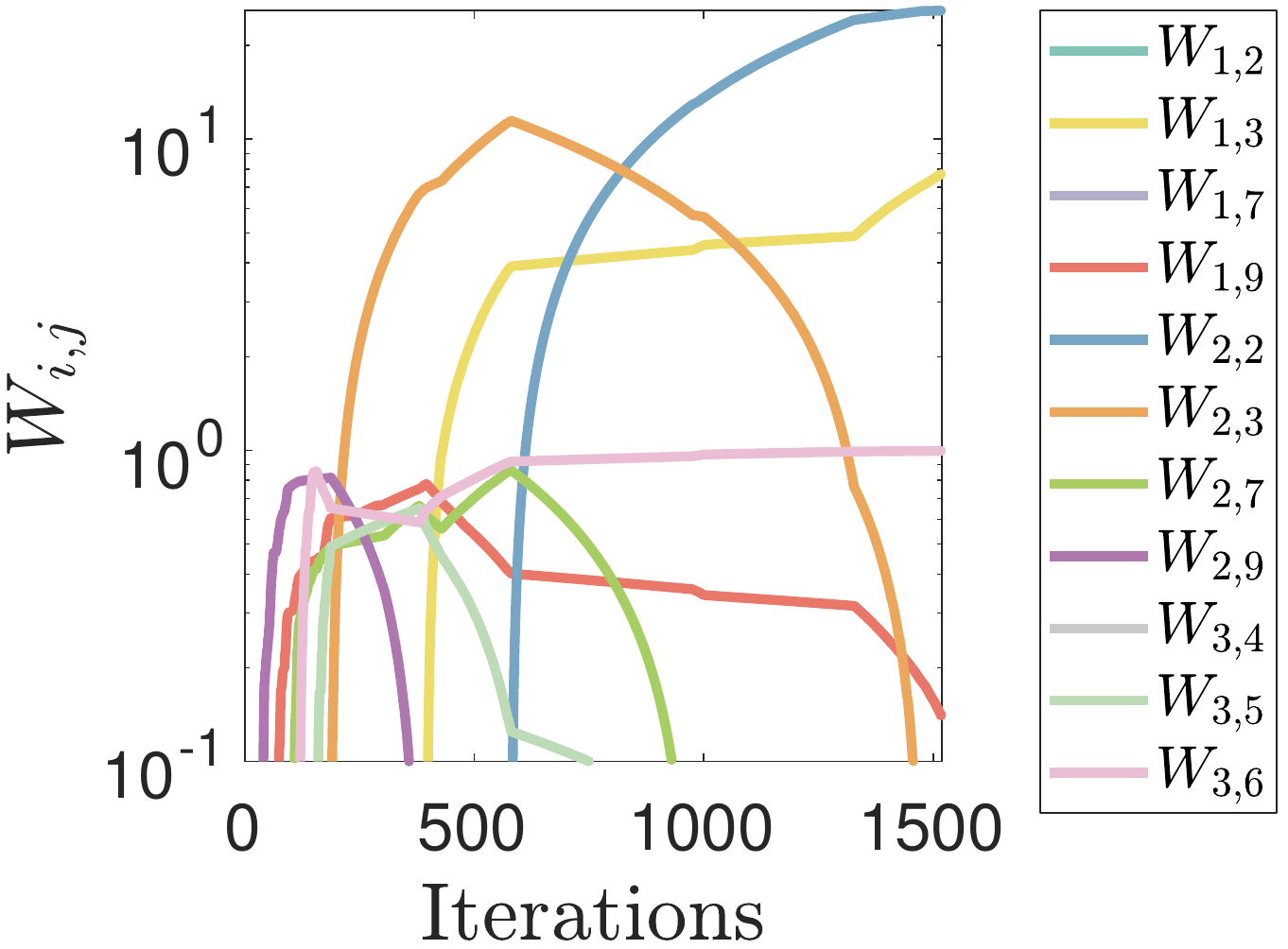} \\ (b) } \ \\
	\parbox[b]{0.49\textwidth}{\centering \includegraphics[width=0.49\textwidth]{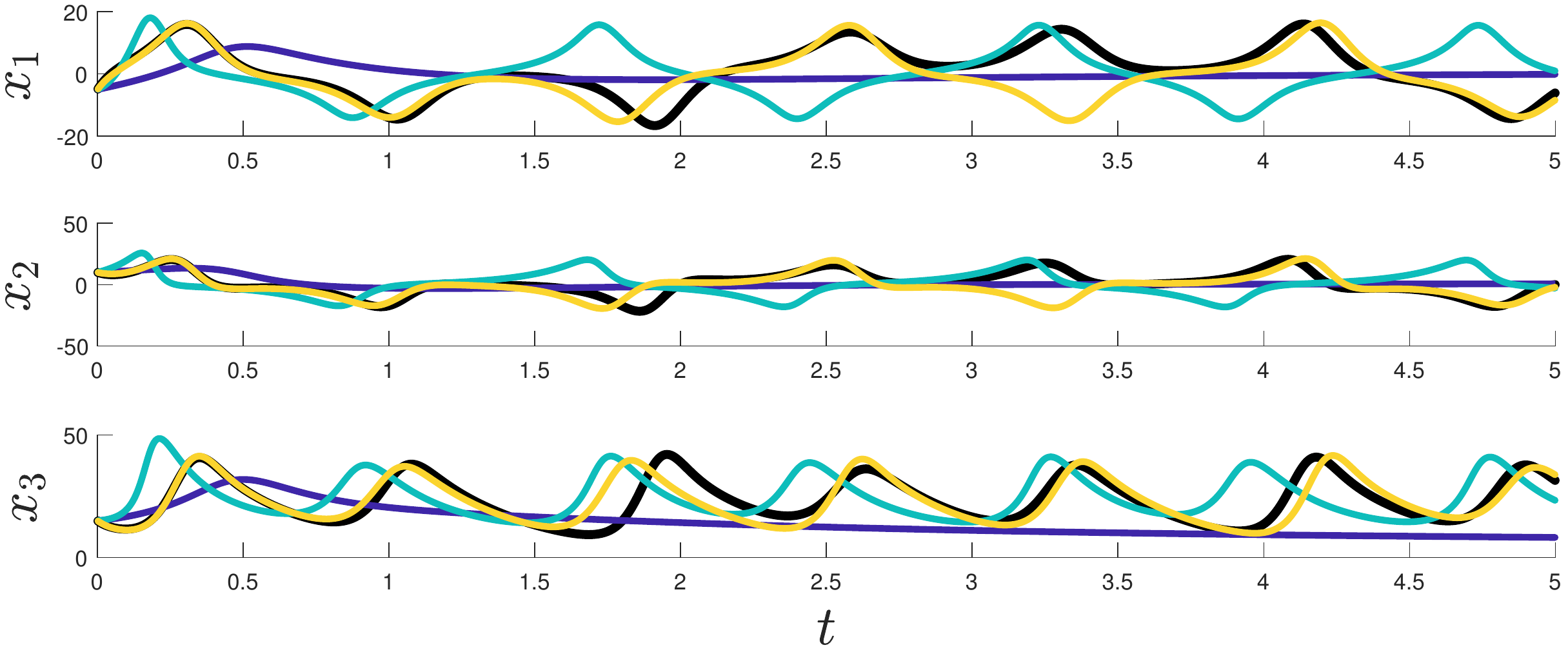} \\ (c) }
	\caption{(a) Convex Pareto front for the SINDy approach applied to the Lorenz system.
	(b) Evolution of all $W_{i,j}$ which become greater than 0.5 over the iterations. (c) Trajectory of the Lorenz system (black) and the resulting trajectories of the approximated solutions where the corresponding point on the Pareto front is marked in (a).}
	\label{fig:Lorenz}
\end{figure}

The Pareto front of the corresponding MOP -- obtained via continuation -- is shown in Figure \ref{fig:Lorenz}(a), the corresponding non-zero weights and three trajectories are visualized in (b) and (c), respectively. We see that although the Pareto front is convex, choosing the weight in a penalization approach can be challenging. Large parts of the Pareto front are almost parallel to either the horizontal or vertical axis, such that a small change in the penalty weight will result in large jumps along the front (cf.~also Figure~\ref{fig:weighted_sum}).
Note that the so-called \emph{knee} point is not necessarily the desired compromise, as the system corresponding to the turquoise point is significantly less accurate on unseen data than the solution corresponding to the yellow point.

\subsection{Neural Network}\label{subsec:neural_network}
Our next example is the training process of the weights $W$ of a feed-forward neural network (NN), i.e., $y=g_W(x)$ \cite{Goodfellow2016}:
\begin{equation*}
    \begin{aligned}
        z_0 &= x, \\
        z_{j + 1} &= h\left( w_j^\top z_j + w_j^0 \right) \quad \mbox{for }j=0,\ldots,K-1, \\
        y &= \sigma(z_K),
    \end{aligned}
\end{equation*}
where $K$ is the number of \emph{hidden layers} and $\sigma$ and $h$ are activation functions. The set of weight matrices, i.e., $W = \{(w_1^0, w_1) \ldots, (w_{K-1}^0, w_{K-1})\}$ (which can again be transformed into a single weight vector), is chosen such that the error on a given training data set $((x^1,y^1),\ldots,(x^N,y^N))$ is minimized. In the case of a classification problem, the error is expressed as the mean of the cross entropy loss, which results in the following objective function:
\begin{equation*}
    f(W) = -\frac{1}{N}\sum_{i = 1}^N \sum_{c=1}^M (y^i_c\log(g_W(x^i)_{c})),
\end{equation*}
where $M$ is the number of classes in the data set and $y_c^i \in \{0,1\}$ indicates whether $x^i$ belongs to class $c$. 
In addition, sparsity is often desired to increase robustness against noise.

We here consider a feed forward neural network with $K=2$ hidden layers to solve the classification task of the well-known Iris data set \cite{fischer1936}. The activation function in the hidden layers is $h(x) = \tanh(x)$, and we chose the softmax function for the output layer, i.e., $\sigma(x)_i = \frac{\exp(x)_i}{\sum_{j=1}^m(\exp(x)_j)}$. Each layer consists of two neurons such that the total number of weights is $|W|=25$. 

\begin{figure}[h!]
	\centering \includegraphics[width=0.44\textwidth]{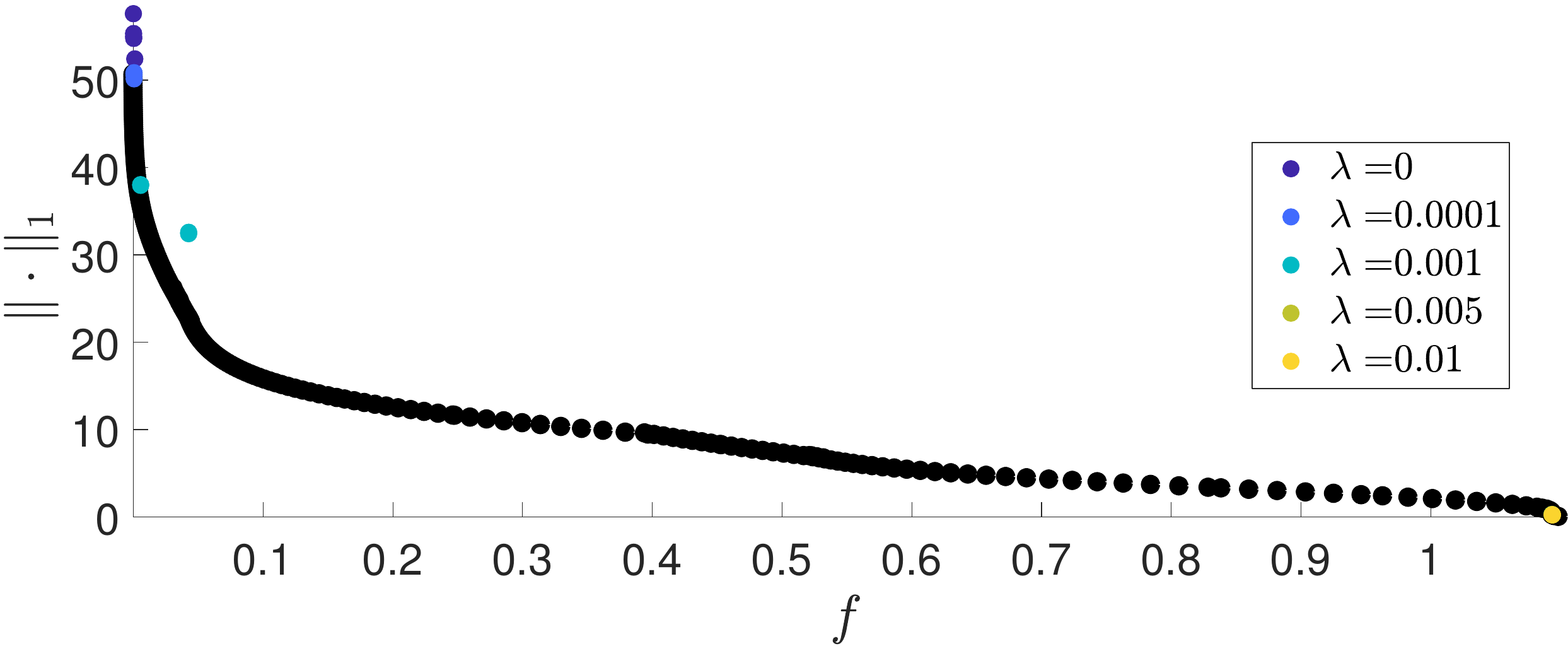} \ \\
	\includegraphics[width=0.22\textwidth]{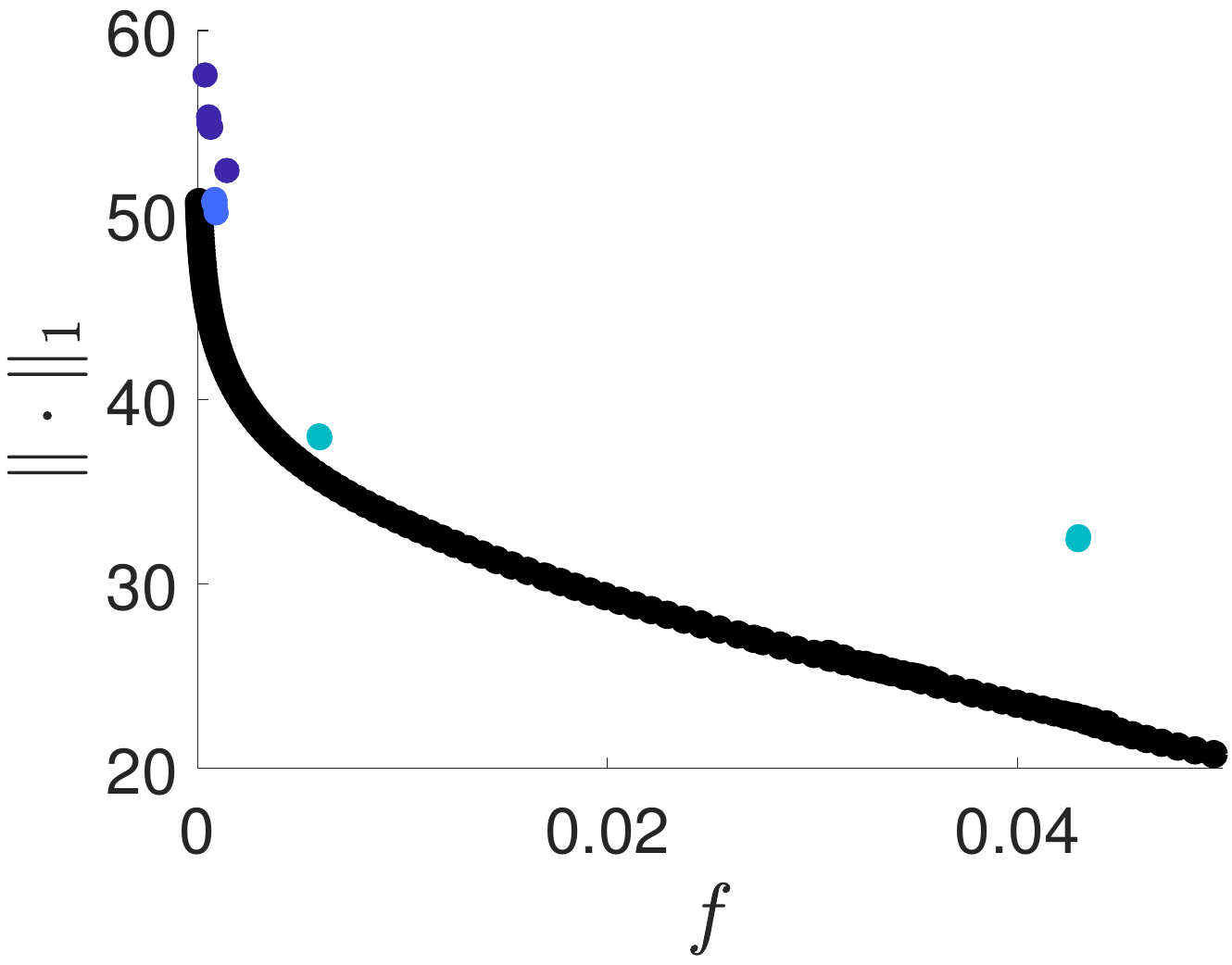} 
	\includegraphics[width=0.22\textwidth]{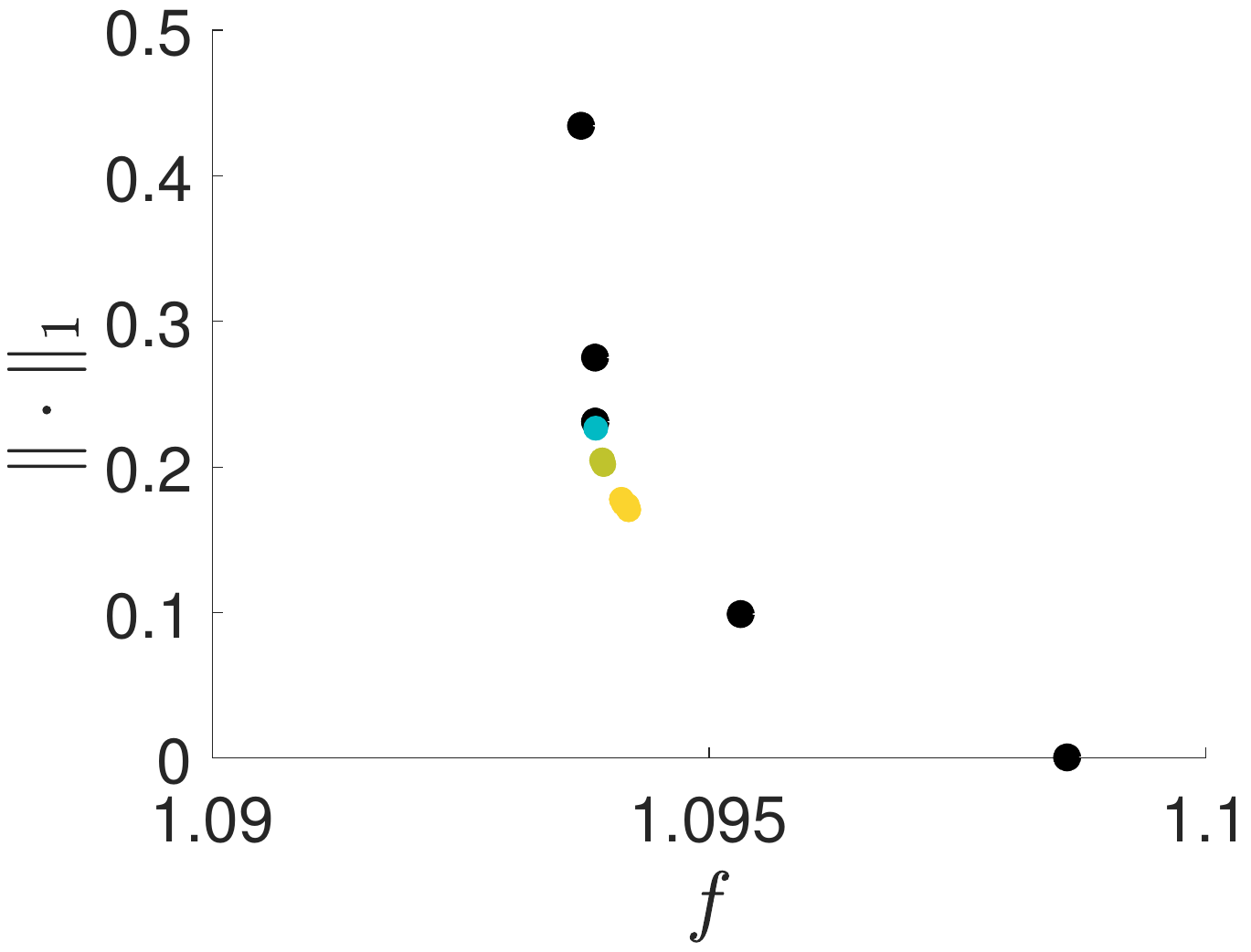} 
	\caption{Pareto front for the NN training on the Iris data set (black). The colored dots show solutions obtained via Adam with different weights for the $\ell_1$ penalty term. These cluster in two places, cf.\ the zoomed in plots in the bottom.
	}
	\label{fig:NeuralNetwork}
\end{figure}

Figure~\ref{fig:NeuralNetwork} shows the Pareto front obtained via the continuation method when starting in $W=0$ and taking $\tau = 0.1$ as step size. We computed three different connected components of the Pareto critical set and applied a non-dominance test afterwards. Since the images of the connected components intersect, they do not appear as distinct connected components in the objective space.
Note that large parts of the front are non-convex.
To compare the solution to stochastic gradient descent training, we also show several solutions obtained using the Adam descent method \cite{KB14} with 5000 epochs and random initial conditions as well as varying weights $\lambda$ for the penalty term $\|W\|_1$. When choosing $\lambda$ too large, the solution always tends towards zero (which is likely due to the non-convexity and the convergence to a local optimum). For small values of $\lambda$ (including $\lambda=0$), we obtain points close to the Pareto front. However, these have larger values of $\|W\|_1$ for a comparable training error.
More critically, for $\lambda = 10^{-3}$ we obtain solutions in both clusters, depending on the initial condition. Finally, solutions with intermediate sparsity are never obtained using Adam, as these appear to be -- in the dynamical systems sense -- less attractive for gradient-descent schemes.
This demonstrates the difficulties of the standard penalization approach and the clear advantage of the continuation method, which produces high-quality trade-off solutions right away. 
This is further emphasized when considering the test error based on data that was not used for training.
Figure~\ref{fig:test_error_iris} shows the test error for the solutions computed with the continuation method and Adam. 
While both approaches are capable of computing solutions with a low training error, we observe that the larger sparsity of the continuation method is significantly more robust in terms of the test error.
\begin{figure}
    \centering
    \includegraphics[width=0.44\textwidth]{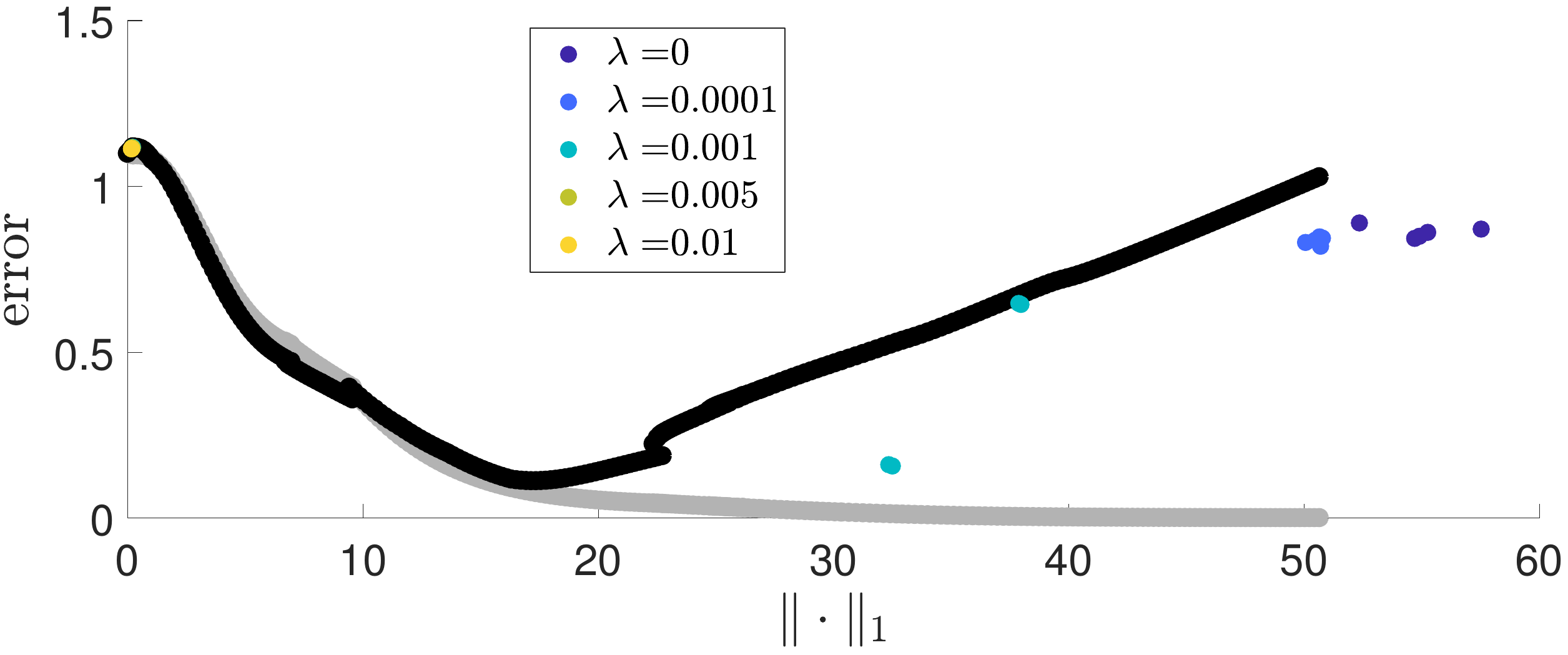}
    \caption{Test error (black) and training error (grey) against $\ell_1$-norm for Pareto critical points computed via continuation. The colored dots show the test error of the Adam solutions.}
    \label{fig:test_error_iris}
\end{figure}

\section{Towards high-dimensional problems} \label{sec:high_dimensional}
Except for the polynomial toy example, the experiments presented in the last section were of low to moderate dimensions. 
In this section, we discuss the general scalability of our method towards higher dimensions and discuss several drawbacks and approaches to resolve some of these. The greatest challenges are related to symmetries, which is a well-known problem in continuation methods \cite{Allgower1990} as well as in other training approaches for neural networks \cite{GB10}.\\
\textbf{Open problems:}
\begin{enumerate}
    \item The computation of the (reduced) Hessian and its inverse are very expensive with increasing dimension $n$.
    \item The necessary conditions for the activation of indices introduced in Section~\ref{subsubsec:activate} often yield a unique direction in which to continue, but there are also cases where several directions remain and the Pareto critical set splits into multiple branches, see also the example in Section~A of the supplementary material. In particular, this is the case if $f$ possesses symmetries, as is often the case for neural networks.
    \item The Hessian may become singular in particular situations, for instance when symmetries occur.
    \item With increasing dimension, the $\varepsilon$-constraint method which we use to find a new component, becomes more expensive and sometimes numerically unstable.
\end{enumerate}
In order to address issue 1), we can use a BFGS or SR1 update (cf. \cite{Nocedal2006}) during the optimization in the Corrector step to obtain an approximation of the Hessian for the predictor step. Furthermore, this way we can directly compute the inverse by a rank-1-update. Experiments have shown that the SR1 update is more suitable in our case.
To handle problem 2), we can apply a greedy strategy, i.e., we proceed only in the first suitable direction and ignore the others. This is reasonable in particular for splitting caused by symmetry, as the branches of $\Pc$ then correspond to the same part of the Pareto front. This way, we can also avoid most of the cases where the Hessian is singular (issue 3)). Nevertheless, it is an open problem how to continue when the Hessian matrix is not regular. The only way to solve problem 4) is to not switch the component. Therefore, one could compute only the \emph{relevant} connected component by computing the appropriate $x_{\textnormal{start}}$ using scalarization (e.g., weighted sum solved with Adam). 
Alternatively, additional knowledge (e.g., regarding the influence of bias neurons on the number of connected components in a neural network) may help in finding suitable transitions to neighboring connected components.
In the following example, we apply the above suggestions to train a larger neural network.
\begin{example}[NN for reduced MNIST]
As in Section~\ref{subsec:neural_network}, we train a feed-forward neural network. We consider the well-known MNIST data set \cite{lecun2010mnist}. The data set consists of images with $28 \times 28$ pixels which results in a large number of trainable parameters even for few neurons. Therefore, we only take images labeled with "$3$" or "$6$" and reduce the number of pixels to $6\times6$, similar to \cite{TF_reduceMNIST}. To classify these images, we use a neural network with $K=2$ hidden layers and four neurons per layer, which results in $\abs{W}=173$ parameters to train. We choose the softplus activation function $h(x) = \log(\exp(x) + 1)$.
Figure~\ref{fig:NeuralNetwork_larger}(a) shows the Pareto front obtained via the continuation method when starting in $W=0$ with $\tau = 0.25$. Similar to Section~\ref{subsec:neural_network}, we computed three components and used a non-dominance test to obtain the results. The shown Adam solutions are obtained with $500$ epochs. 
\begin{figure}[h!]
	\centering
	\parbox[b]{0.22\textwidth}{\centering \includegraphics[width=0.22\textwidth]{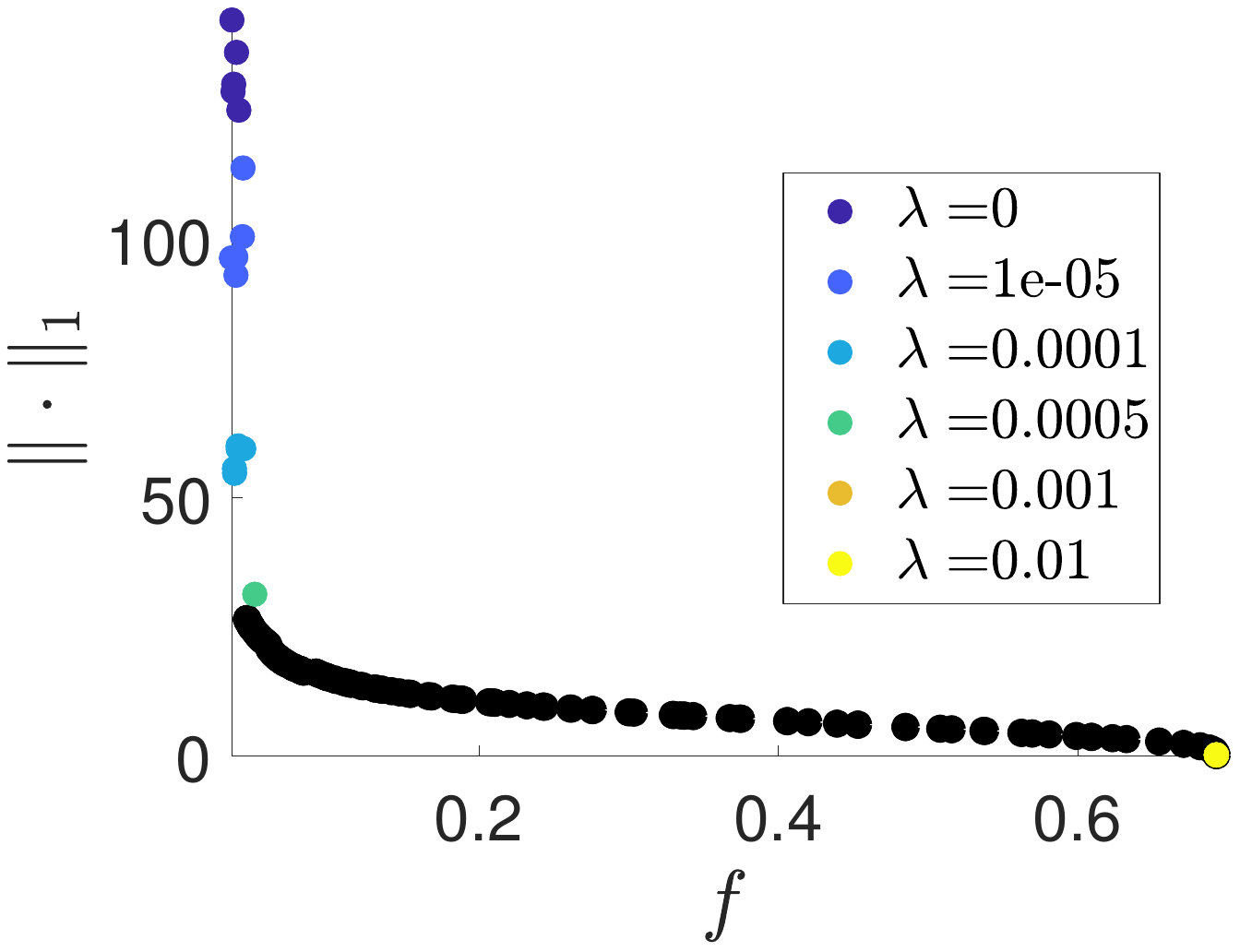}\\ (a) } \hfil
	\parbox[b]{0.22\textwidth}{\centering \includegraphics[width=0.22\textwidth]{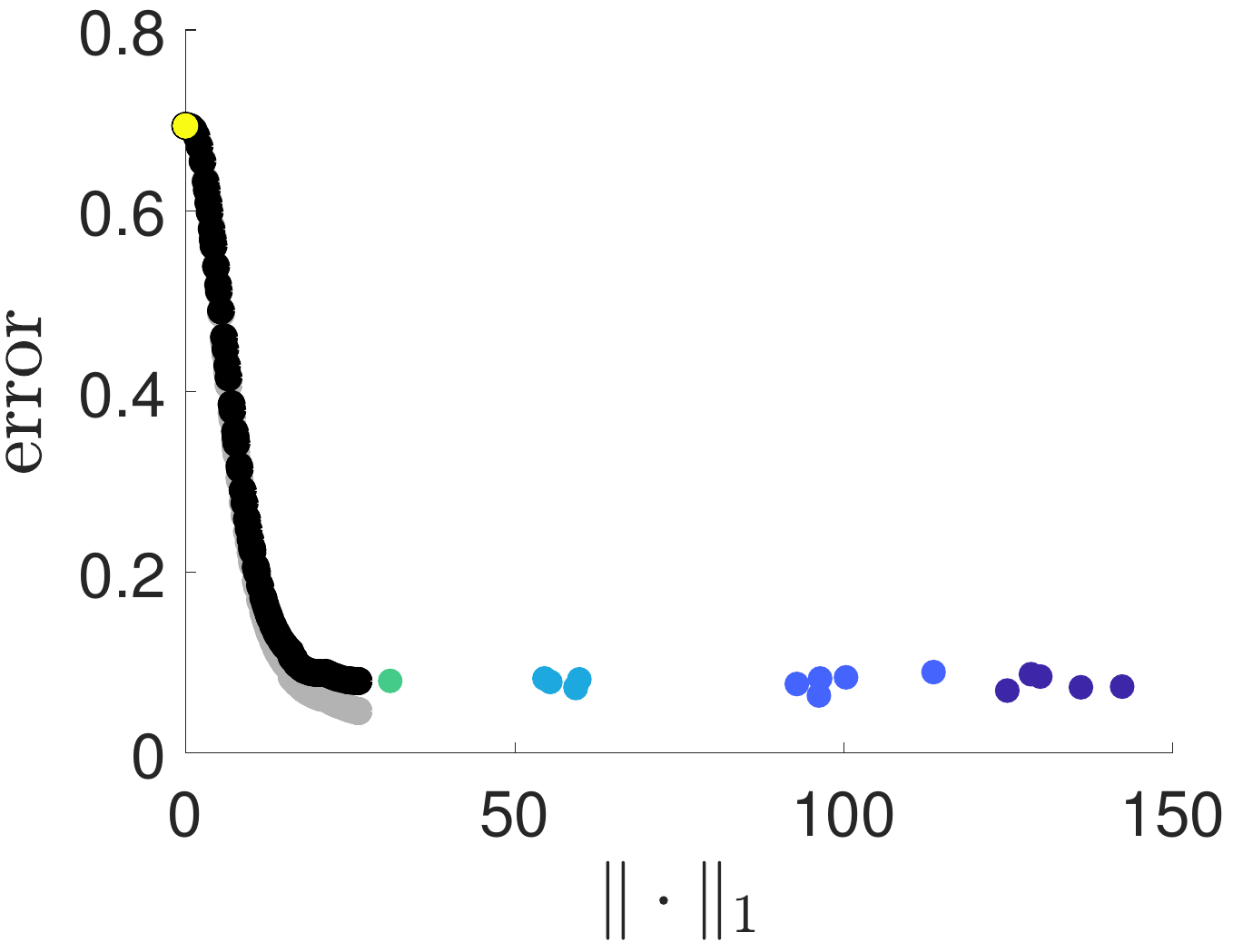}\\ (b)}
	\caption{(a) Pareto front for the NN training on the reduced MNIST data set (black). The colored dots show solutions obtained via Adam with different weights for the $\ell_1$-term. (b) Test error (black) and training error (grey) against the $\ell_1$-norm for Pareto critical points computed via continuation. The colored dots show the test error of the Adam solutions.}
	\label{fig:NeuralNetwork_larger}
\end{figure}
We observe again that they cluster in two places and that they are significantly less sparse. Furthermore, the Adam solutions suggest that there exists another part of the Pareto front with significantly less sparsity but a slightly lower error, which is missing in our solution and would require additional runs with additional $x_{\text{start}}$. However, as can be seen in Figure~\ref{fig:NeuralNetwork_larger}(b), the effect on both training and test error is almost negligible and thus, these solutions are not of particular practical relevance.
Observe that the test error does not increase again, as the best fitting model is not sparse.
\end{example}
\section{Conclusion}\label{sec:conclusion}
The continuation method we present in this paper allows researchers and practitioners with an interest in sparse optimization to systematically calculate the entire set of compromise solutions between sparsity and the main objective, thereby allowing for much more insight as well as informed model selection than weighted-sum approaches or descent methods providing single Pareto optima \cite{KB14,GP21a}.
The presented approach extends the well-known homotopy methods to the much more complex nonlinear problem setting, where weighting approaches fail due to non-convexity. The optimality conditions specifically derived for this problem class allow us to use continuation methods from smooth multiobjective optimization almost everywhere, and to efficiently analyze and exploit the problem structure in points where non-smoothness occurs. 

To unfold the full potential of this multiobjective approach to sparse optimization, several additional questions should be addressed in future work. 
These contain, for instance, how to deal with symmetries, in particular with regard to higher-dimensional problems. This is of particular interest for the training of neural networks with many training parameters.
Furthermore, in order to further extend the range of possible applications, the consideration of additional -- more general non-smooth -- objectives as well as constraints (see, for instance, \cite{LB18}) is certainly of high interest.
Finally, an extension to more than one smooth objective may be useful.
This could be done by combining our results with the approach in \cite{SCM+20}.

\ifCLASSOPTIONcompsoc
  \section*{Acknowledgments}
\else
  \section*{Acknowledgment}
\fi

This research has been funded by the European Union and the German Federal State of North Rhine-Westphalia within the EFRE.NRW project ``SET CPS'', and by the DFG Priority Programme 1962 ``Non-smooth and Complementarity-based Distributed Parameter Systems''.

\ifCLASSOPTIONcaptionsoff
  \newpage
\fi



\bibliographystyle{IEEEtran}
%
\bibliography{references}
%
%
%
%
\begin{IEEEbiography}[{\includegraphics[width=1in,height=1.25in,clip,keepaspectratio]{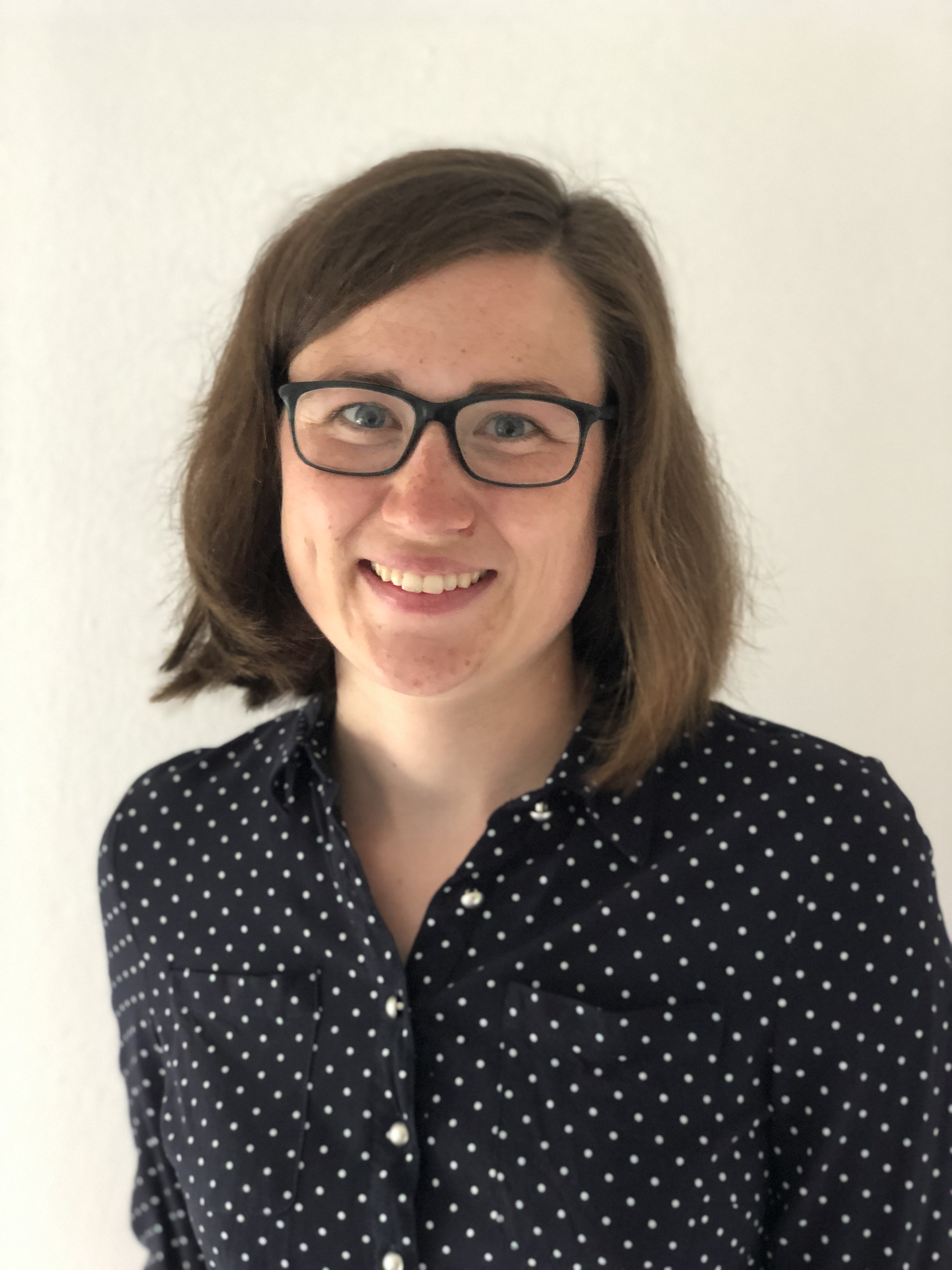}}]{Katharina Bieker}
received the B.Sc. and M.Sc. degree in technomathematics from Paderborn University, Germany, in 2015 and 2017, respectively. She is currently a research assistant at the Chair of Applied Mathematics at Paderborn University and is studying towards a PhD degree with a main focus on data-driven modelling and multiobjective optimization.
\end{IEEEbiography}
\vspace*{-1.0cm}
\begin{IEEEbiography}[{\includegraphics[width=1in,height=1.25in,clip,keepaspectratio]{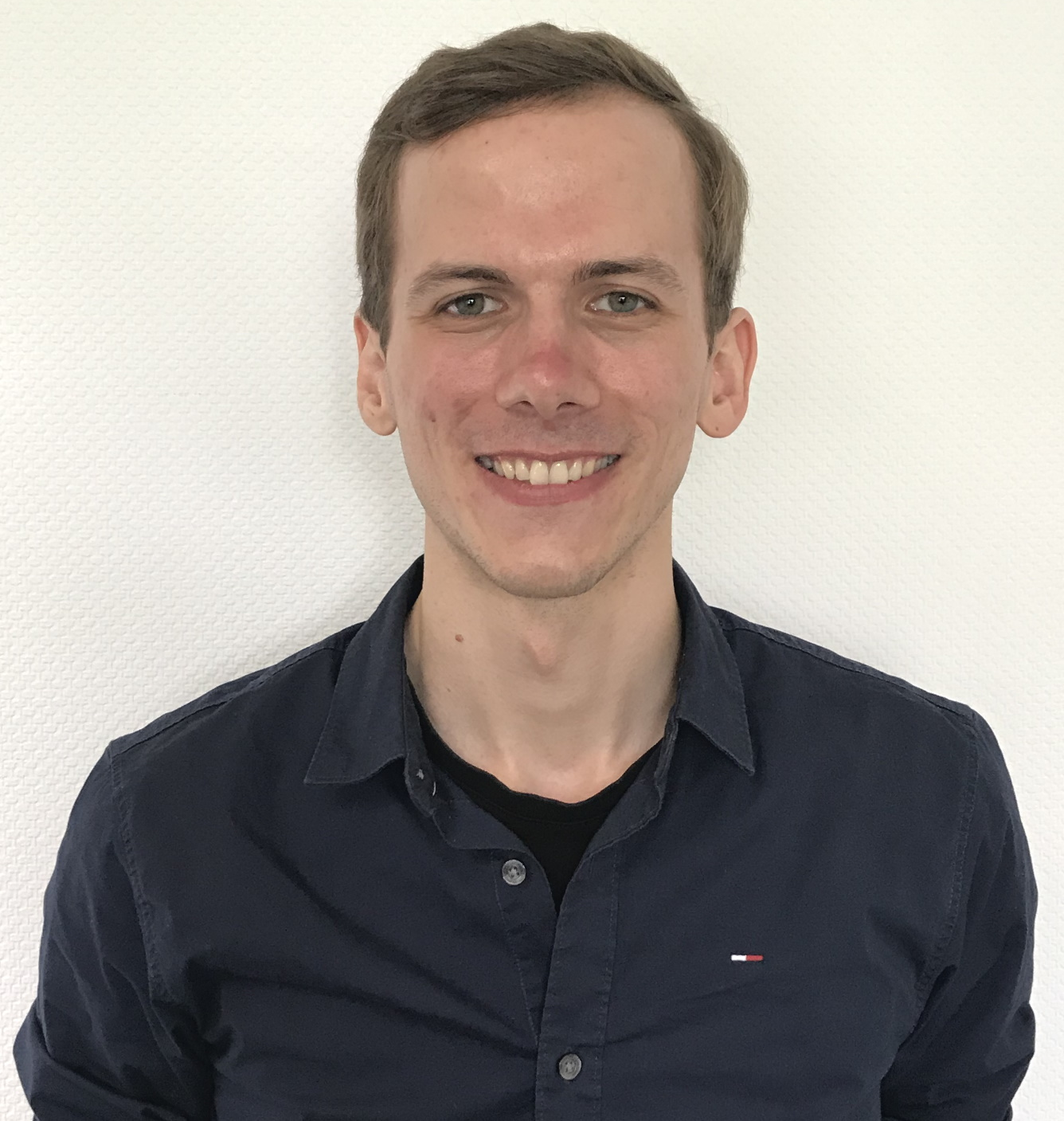}}]{Bennet Gebken}
received the B.Sc. and M.Sc. degree in mathematics from Paderborn University, Germany, in 2014 and 2017, respectively. He is currently a research assistant at the Chair of Applied Mathematics at Paderborn University and is studying towards a PhD degree with a main focus on the structure and computation of Pareto critical sets in multiobjective optimization.
\end{IEEEbiography}
\vspace*{-1.0cm}
\begin{IEEEbiography}[{\includegraphics[width=1in,height=1.25in,clip,keepaspectratio]{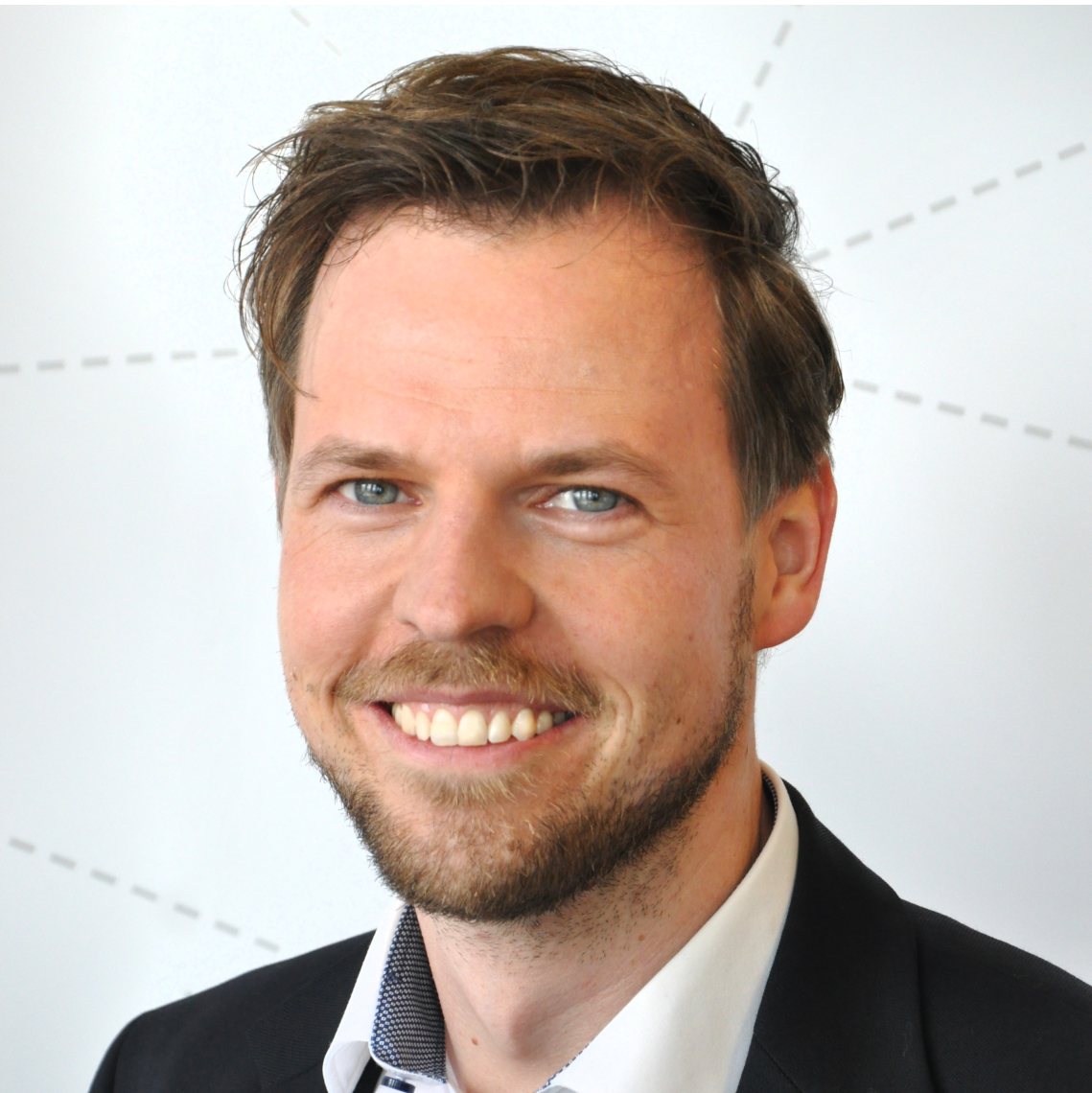}}]{Sebastian Peitz}
received his B.Sc.\ and M.Sc.\ degrees in mechanical engineering from RWTH Aachen University, Germany, in 2011 and 2013, respectively, and his PhD in Mathematics from Paderborn University in 2017. 
He is currently assistant professor for ``Data Science for Engineering'' in the Department of Computer Science at Paderborn University. His research interests are multiobjective optimization, optimal control and data-driven modeling of complex systems.
\end{IEEEbiography}
%
%
%
%
%
%
%
\end{document}


%
\title{Supplementary Material:\\ On the Treatment of Optimization Problems with L1 Penalty Terms via Multiobjective Continuation}
%
%
%
%

\author{Katharina Bieker,
        Bennet Gebken
        and~Sebastian Peitz
\IEEEcompsocitemizethanks{\IEEEcompsocthanksitem K. Bieker and B. Gebken are with the Department
of Mathematics, Paderborn University, Germany.\protect\\
E-mail: bieker@math.upb.de\protect\\
\IEEEcompsocthanksitem S. Peitz is with the Department of Computer Science, Paderborn University, Germany.
}%
\thanks{~}}

%
%

\markboth{~}%
{Shell \MakeLowercase{\textit{et al.}}: Bare Demo of IEEEtran.cls for Computer Society Journals}
%



\IEEEtitleabstractindextext{%

}

\maketitle

\IEEEdisplaynontitleabstractindextext

%
\IEEEpeerreviewmaketitle

\renewcommand{\thesection}{\alph{section}}
\section{An example for the splitting of Pareto critical sets}

In order to study the issue of splitting Pareto sets in points where the activation structure changes (cf.\ Sec.\ 3.4.1 in the main document), we consider the MOP 
\begin{equation*}
\min_{x \in \R^3}\left( \begin{matrix} (x_1 - 2.5)^2 + ((x_2+1)^2 + (x_3-1)^2 - 3)^2 \\ \|x\|_1 \end{matrix} \right).
\end{equation*}
The Pareto critical set $\Pc$ and its image are presented in Figure~\ref{fig:example_split}. We see that in the point $x=(0.5,0.0,0.0)$, the Pareto critical set splits into three parts visualized in different colors. In the blue part, only indices one and three are active, two and three are active in the turquoise part, whereas all variables are active in the yellow part. Obviously, this is caused by the symmetry in $x_2$ and $x_3$, i.e., by the fact that $f((x_1,x_2,x_3)) = f((x_1,-x_3,-x_2))$.
Note, that the image of all three parts is equal (Figure~\ref{fig:example_split} (b)).

\begin{figure}[h!]
\centering
    \parbox[b]{0.48\textwidth}{\centering \includegraphics[width=0.35\textwidth]{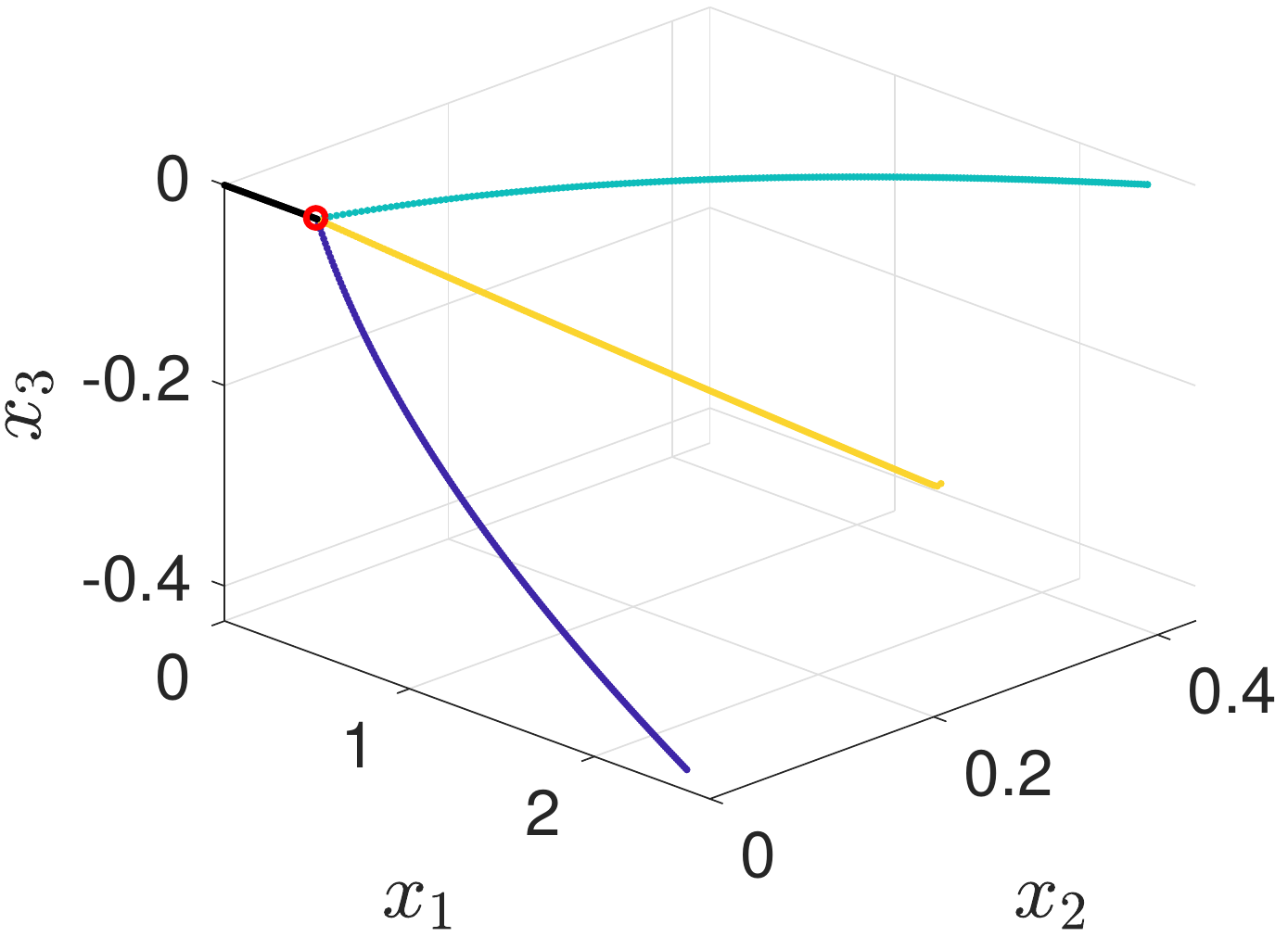} \\ (a)}
    \parbox[b]{0.48\textwidth}{\centering\includegraphics[width=0.35\textwidth]{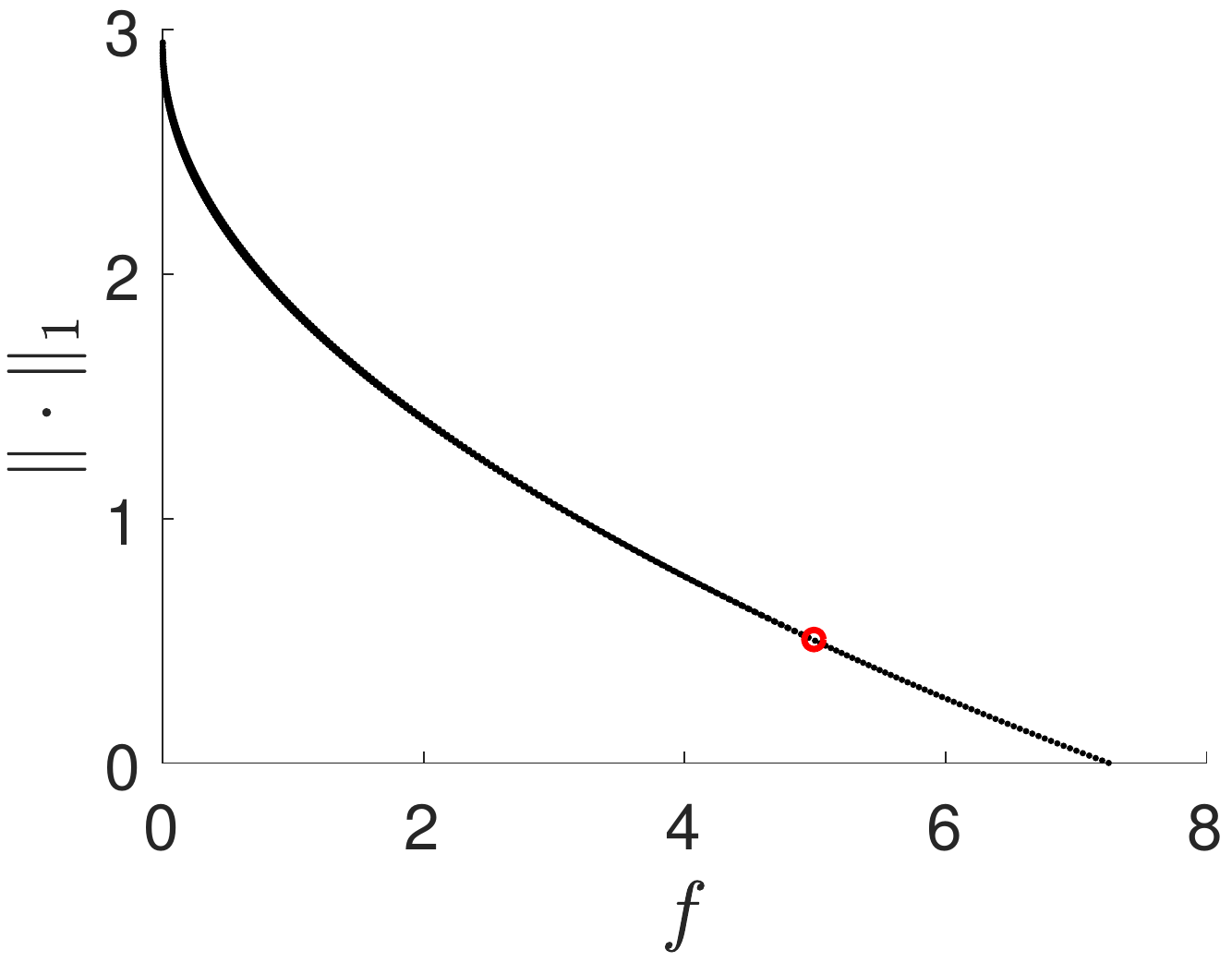}\\(b)}
    \caption{(a) Pareto critical set $\Pc$ of the example. In the red point, $\Pc$ splits into the blue, turquoise and yellow part. (b) The image of $\Pc$, with the image of the point where $\Pc$ splits shown in red.}
    \label{fig:example_split}
\end{figure}

\section{Proof of Corollary 3.3}
\begin{manualcorollary}{3.3}
Let $x^0 \neq 0$ be a Pareto critical point such that the inequality in 2(b) of Theorem 3.1 is strict. Let $j_1 < ... < j_{n^A(x^0)}$ be the (sorted) active indices at $x^0$. Define 
	the following two mappings that lift the active components to the $n$-dimensional point $x$ and vice versa:
	\begin{align*}
	&h^1 : \R^{n^A(x^0)} \rightarrow \R^n, \quad h^1_j(\bar{x}) = 
	\begin{cases}
	0 & \text{if } j \notin A(x^0), \\
	\bar{x}_i & \text{if } j_i \in A(x^0),
	\end{cases} \\
	&h^2 : \R^n \rightarrow \R^{n^A(x^0)}, \quad h^2_i(x) = x_{j_i} \quad \forall i \in \{1,...,n^A(x^0)\}.
	\end{align*}
	Then there is an open set $U \subseteq \R^n$ with $x^0 \in U$ such that $x \in U$ being Pareto critical for (\emph{MOP}-$\ell_1$) is equivalent to $A(x) = A(x^0)$ and $h^2(x)$ being Pareto critical for the MOP
	\begin{align} \label{eq:MOP_l1_active}
	\min_{\bar{x} \in \R^{n^A(x^0)}} \left( \begin{array}{c}
	f(h^1(\bar{x}))  \\
	\norm{\bar{x}}_1
	\end{array}\right).
	\end{align}
\end{manualcorollary}

\begin{proof}
	By definition of $A(x^0)$, there has to be some $U^1 \subseteq \R^n$ with $x^0 \in U^1$ such that
	\begin{align*}
	\sgn(x_j) = \sgn(x_j^0) \neq 0 \quad \forall j \in A(x^0), \ x \in U^1.
	\end{align*}
	In particular,
	\begin{align} \label{eq:cor_MOP_active_U1}
	A(x^0) \subseteq A(x) \quad \forall x \in U^1.
	\end{align}
	Let $a \in A(x^0)$. Since $\nabla f$ is continuous and since the inequality in 2(b) of Theorem 3.1 is strict, there has to be an open set $U^2 \subseteq \R^n$ with $x^0 \in U^2$ such that for all $x \in U^2$
	\begin{align*}
	&\abs{(\nabla f(x))_j} < \abs{(\nabla f(x^0))_a} \ \forall j \notin A(x^0) \quad \text{and} \\
	&\min_{j \notin A(x^0)} \abs{(\nabla f(x))_j} < \min_{j \in A(x^0)} \abs{(\nabla f(x))_j}.
	\end{align*} 
	This implies
	\begin{align} \label{eq:cor_MOP_active_U2}
	A(x) \subseteq A(x^0) \quad \forall x \in U^2 \cap \Pc.
	\end{align}
	Now let $U = U^1 \cap U^2$ and $x \in U$. By \eqref{eq:cor_MOP_active_U1}, \eqref{eq:cor_MOP_active_U2} and Theorem 3.1, $x$ being Pareto critical for (MOP-$\ell_1$) is equivalent to
	\begin{align*}
	A(x) &= A(x^0), \\
	\sgn((\nabla f(x))_j) &= -\sgn(x_j) \quad \forall j \in A(x^0), \\
	\abs{(\nabla f(x))_j} &= \abs{(\nabla f(x))_a} \quad \forall j \in A(x^0).
	\end{align*}
	By applying Theorem 3.1 to the MOP \eqref{eq:MOP_l1_active} (which has the same form as (MOP-$\ell_1$)), we see that the last two lines of the previous system of equations are precisely the conditions for $h^2(x)$ being Pareto critical. This completes the proof. \\
\end{proof}

\section{Proof of Lemma 3.5}

\begin{manuallemma}{3.5}
	Let $x^* \in P_c \setminus \{ 0 \}$ with $\nabla f(x^*) \neq 0$. Let $\gamma : (-1,1) \rightarrow \R^n$ such that
	\begin{itemize}
		\item[(i)] $\gamma$ is continuously differentiable,
		\item[(ii)] $\gamma(t) \in P_c \quad \forall t \in [0,1)$, 
		\item[(iii)] $\gamma(0) = x^*$,
		\item[(iv)] $A(\gamma(t)) =: A_\gamma$ is constant $\forall t \in (0,1)$,
	\end{itemize}
	Let $a \in A(x^*)$ and $w = \gamma'(0)$. Then for all $j \in A_\gamma \setminus A(x^*)$ we have
	\begin{enumerate}[label=(\alph*)]
		\item either $w_j = 0$ or $\sgn(w_j) = - \sgn((\nabla f(x^*))_j)$,
	\end{enumerate}
	and for all $j \in A_p(x^*)\setminus A_\gamma$ we get
	\begin{enumerate}[resume, label=(\alph*)]
		\item \quad\,$\sgn((\nabla f(x^*))_j) (\nabla^2 f(x^*))_j w \le \sgn((\nabla f(x^*))_a) (\nabla^2 f(x^*))_a w$,
	\end{enumerate}
	where $(\nabla^2 f(x^*))_j$ is the $j$-th row of $\nabla^2 f(x^*)$.
\end{manuallemma}

\begin{proof}
First, note that due to continuity of $\nabla f$ and $\gamma$, we can assume without loss of generality that
	\begin{align} \label{eq:4}
	(\nabla f(\gamma(t)))_{j} \neq 0 \quad \forall t \in (0,1), \ j \in A_\gamma \cup A_p(x^*).
	\end{align}
	This implies that $\sgn((\nabla f(\gamma(t)))_j)$ is constant for all $t \in [0,1)$ and each $j \in A_\gamma \cup A_p(x^*)$. \\
	(a)  Let $j \in A_\gamma \setminus A(x^*)$.
	From (ii) and (iv) it follows that 
	\begin{align} \label{eq:1}
	\sgn(\gamma(t)_j) = -\sgn((\nabla f(\gamma(t)))_j)
	\end{align}		
	for all $t \in (0,1)$ and each $j \in A_\gamma$. 
	By the above considerations, this implies that $\sgn(\gamma(t)_j)$ does not change, i.e., is constant on $(0,1)$ for each $j \in A_\gamma$.
	By (iii) and since $j \notin A(x^*)$, we have $\gamma_j(0) = x^*_j = 0$, so
	\begin{align*}
	w_j = \gamma'_j(0) = \lim_{t \rightarrow 0} \frac{\gamma_j(t) - \gamma_j(0)}{t} = \lim_{t \rightarrow 0} \frac{\gamma_j(t)}{t}  = \lim_{t \downarrow  0} \frac{\gamma_j(t)}{t}.
	\end{align*}		
	If $w_j \neq 0$, then \eqref{eq:1} implies $\sgn(w_j) = \sgn(\gamma_j(t)) = -\sgn((\nabla f(\gamma(t)))_j)$ for all $t \in (0,1)$. 
	Continuity of $(\nabla f) \circ \gamma$ and $(\nabla f(x^*))_j \neq 0$ (cf. \eqref{eq:4}) imply
	\begin{align*}
	\sgn(w_j) = -\sgn((\nabla f(\gamma(0)))_j) = -\sgn((\nabla f(x^*))_j).
	\end{align*}
	(b) Let $j \in A_p(x^*)\setminus A_\gamma$. From (ii) and 2(b) of Theorem 3.1, it follows that
	\begin{equation} \label{eq:2}
	\begin{aligned}
	&-\sgn((\nabla f(\gamma(t)))_j) (\nabla f(\gamma(t)))_j  =|(\nabla f(\gamma(t)))_j| \\
	\le\ &|(\nabla f(\gamma(t)))_a|= -\sgn((\nabla f(\gamma(t)))_a) (\nabla f(\gamma(t)))_a
	\end{aligned}
	\end{equation}
	for all $t \in (0,1)$. As $j \in A_p(x^*)$, we have $\abs{(\nabla f(x^*))_j} = \abs{(\nabla f(x^*))_a}$ and thus, $\sgn((\nabla f(\gamma(0)))_j) (\nabla f(\gamma(0)))_j = \sgn((\nabla f(\gamma(0)))_a) (\nabla f(\gamma(0)))_a$. Hence, from \eqref{eq:2} we can conclude
	\begin{equation}\label{eq:3}
	\begin{aligned}
	&\frac{\sgn((\nabla f(\gamma(t)))_j) (\nabla f(\gamma(t)))_j}{t} - \frac{\sgn((\nabla f(\gamma(0)))_j) (\nabla f(\gamma(0)))_j }{t} \\\le~ &\frac{\sgn((\nabla f(\gamma(t)))_a) (\nabla f(\gamma(t)))_a}{t} - \frac{\sgn((\nabla f(\gamma(0)))_a) (\nabla f(\gamma(0)))_a }{t}
	\end{aligned}
	\end{equation}
	for all $t \in (0,1)$. Recall that $\sgn((\nabla f(\gamma(t)))_j)$ and $\sgn((\nabla f(\gamma(t)))_a)$ are constant with respect to $t \in [0,1)$. Thus, \eqref{eq:3} is equivalent to
	\begin{equation*}
	\sgn((\nabla f(\gamma(0)))_j) \frac{ (\nabla f(\gamma(t)))_j - (\nabla f(\gamma(0)))_j }{t} \le~ \sgn((\nabla f(\gamma(0)))_a) \frac{(\nabla f(\gamma(t)))_a - (\nabla f(\gamma(0)))_a }{t}
	\end{equation*} 
	for all $t \in (0,1)$. Thus, taking the limit $ t \downarrow 0$ on both sides of \eqref{eq:3} yields
	\begin{align*}
	\sgn((\nabla f(\gamma(0)))_j) (\nabla^2 f(\gamma(0)))_j \gamma'(0) \le\ \sgn((\nabla f(\gamma(0)))_a) (\nabla^2 f(\gamma(0)))_a \gamma'(0),
	\end{align*}
	where $(\nabla^2 f(\gamma(0)))_j$ is the $j$-th row of $\nabla^2 f(\gamma(0))$. Replacing $\gamma(0) = x^*$ and $\gamma'(0) = w$ yields
	\begin{align*}
	\sgn((\nabla f(x^*))_j) (\nabla^2 f(x^*))_j w  \le\ \sgn((\nabla f(x^*))_a) (\nabla^2 f(x^*))_a w.
	\end{align*}
\end{proof}